\documentclass[11pt,reqno]{amsart}

\usepackage{amsmath,amsthm,amssymb,verbatim,easybmat,enumerate,color}
\usepackage[tmargin=1.0in,bmargin=1.0in,rmargin=1.0in,lmargin=1.0in]{geometry}
\usepackage{chngcntr}
\usepackage[breaklinks=true]{hyperref}
\usepackage{tocvsec2}

\theoremstyle{plain}

\newtheorem{theorem}{Theorem}[section]
\newtheorem{corollary}[theorem]{Corollary}
\newtheorem{proposition}[theorem]{Proposition}
\newtheorem{lemma}[theorem]{Lemma}

\theoremstyle{definition}

\newtheorem{definition}[theorem]{Definition}

\newtheorem{remark}[theorem]{Remark}

\theoremstyle{remark}

\numberwithin{equation}{section}

\counterwithin{table}{section}

\makeatletter
\let\c@theorem\c@table
\makeatother

\def\img {\mathop{\rm Im}}

\def\diag {\mathop{\rm diag}}
\def\rk {\mathop{\rm rank}}

\def\sign {\mathop{\rm sign}}

\def\cp{\mathop{\rm CP}}
\def\R{\mathbb{R}}

\def\N{\mathbb{N}}

\newcommand{\bp}{\ensuremath{\mathbb P}}
\newcommand{\pos}{\ensuremath{\mathcal H}_{pos}}
\newcommand{\mono}{\ensuremath{\mathcal H}_{mono}}
\newcommand{\conv}{\ensuremath{\mathcal H}_{conv}}
\newcommand{\super}{\ensuremath{\mathcal H}_{super}}
\newcommand{\sub}{\ensuremath{\mathcal H}_{sub}}

\def\calh{\mathcal H}

\begin{document}
\title[Characterizing Hadamard powers preserving positivity, monotonicity,
and convexity]
{Complete characterization of Hadamard powers preserving Loewner
positivity, monotonicity, and convexity}
\author[Dominique Guillot \and Apoorva Khare \and Bala
Rajaratnam]{Dominique Guillot \and Apoorva Khare \and Bala
Rajaratnam\\Stanford University}
\address{Department of Mathematics, Stanford University, Stanford, CA,
USA}
\date{\today}


\begin{abstract}
Entrywise powers of symmetric matrices preserving positivity,
monotonicity or convexity with respect to the Loewner ordering arise in
various applications, and have received much attention recently in the
literature. Following FitzGerald and Horn [{\it J.~Math.~Anal.~Appl.},
1977], it is well-known that there exists a \emph{critical exponent}
beyond which all entrywise powers preserve positive definiteness. Similar
phase transition phenomena have also recently been shown by Hiai (2009)
to occur for monotonicity and convexity. In this paper, we complete the
characterization of all the entrywise powers below and above the critical
exponents that are positive, monotone, or convex on the cone of positive
semidefinite matrices. We then extend the original problem by fully
classifying the positive, monotone, or convex powers in a more general
setting where additional rank constraints are imposed on the matrices. We
also classify the entrywise powers that are super/sub-additive with
respect to the Loewner ordering. Finally, we extend all the previous
characterizations to matrices with negative entries. Our analysis
consequently allows us to answer a question raised by Bhatia and Elsner
(2007) regarding the smallest dimension for which even extensions of the
power functions do not preserve Loewner positivity.
\end{abstract}
\maketitle

\section{Introduction and main results}

The study of positive definite matrices and of functions that preserve
them arises naturally in many branches of mathematics and other
disciplines. Given a function $f: \R \to \R$ and a matrix $A = (a_{ij})$,
the matrix $f[A] := (f(a_{ij}))$ is obtained by applying $f$ to the
entries of $A$. Such mappings are called \emph{Hadamard functions} (see
\cite[\S 6.3]{Horn_and_Johnson_Topics}) and appear naturally in many
fields of pure and applied mathematics, probability, and statistics.

Characterizing functions $f: \R \to \R$ such that $f[A]$ is positive
semidefinite for every positive semidefinite matrix $A$ is critical for
many applications. For example, in modern high-dimensional probability
and statistics, functions are often applied to the entries of
covariance/correlation matrices in order to obtain regularized estimators
with attractive properties (like sparsity, good condition number, etc.).
Particular examples of functions used in practice include the so-called
\emph{hard} and $\emph{soft-thresholding}$ functions (see
\cite{bickel_levina, Guillot_Rajaratnam2012, Guillot_Rajaratnam2012b,
hero_rajaratnam, Hero_Rajaratnam2012}), and the power functions - see
e.g.~\cite{Li_Horvath} and \cite[\S 2.2]{Zhang_Horvath}. The resulting
matrices often serve as ingredients in other statistical procedures that
require these matrices to be positive semidefinite. In order for such
procedures to be widely applicable, it is therefore important to know
whether a given Hadamard function preserves positivity.

Let $A$ be a positive semidefinite matrix with nonnegative entries. In
this paper, we study the properties of entrywise powers of $A$, i.e., the
properties of $f[A]$ when $f(x) = x^\alpha$ is applied elementwise to $A$
(for some $\alpha \geq 0$). This question of which entrywise (or
Hadamard) powers $x^\alpha$ preserve Loewner positivity has been widely
studied in the literature. One of the earliest works in this setting is
by FitzGerald and Horn \cite{FitzHorn}, who studied the set of entrywise
powers preserving Loewner positivity among $n \times n$ matrices, in
connection with the Bieberbach conjecture. They show that a certain phase
transition occurs at $\alpha = n-2$. More precisely, every $\alpha \geq
n-2$ as well as every positive integer preserve Loewner positivity, while
no non-integers in $(0,n-2)$ do so.

The phase transition at the integer $n-2$ has been popularly referred to
in the literature as the ``critical exponent" (CE) for preserving Loewner
positivity. Indeed, the study of critical exponents - and more generally
of functions preserving a form of positivity - is an interesting and
important endeavor in a wide variety of situations, and has been studied
in many settings (see e.g.~\cite{Loewner34, johnson_et_al_2011,
Guillot_Khare_Rajaratnam-CEC, micchelli_et_al_1979, Walch_survey}).

While it is more common in the critical exponents literature to study
matrices with nonnegative entries, positive semidefinite matrices
containing negative entries also occur frequently in practice. In recent
work, Hiai \cite{Hiai2009} extended previous work by FitzGerald and Horn
by considering the odd and even extensions of the power functions to
$\R$. Recall that for $\alpha \in \R$, the even and odd multiplicative
extensions to $\R$ of the power function $f_\alpha(x) := x^\alpha$ are
defined to be $\phi_\alpha(x) := |x|^\alpha$ and $\psi_\alpha(x) :=
\sign(x) |x|^\alpha$ at $x \neq 0$, and $f_\alpha(0) = \phi_\alpha(0) =
\psi_\alpha(0) := 0$.
In {\it loc.~cit.}, Hiai \cite{Hiai2009} studied the powers $\alpha > 0$
for which $\phi_\alpha$ and $\psi_\alpha$ preserve Loewner positivity,
and showed that the same phase transition also occurs at $n-2$ for
$\phi_\alpha, \psi_\alpha$, as demonstrated in \cite{FitzHorn}. He also
analyzed functions that are \emph{monotone} and \emph{convex} with
respect to the Loewner ordering, and proved several deep results and
connections between these classes of functions. These results are akin to
the corresponding connections between positivity, monotonicity, and
convexity for real functions of one variable. Before recalling these
notions, we first introduce some notation. Given $n \in \N$ and $I
\subset \R$, let $\bp_n(I)$ denote the set of symmetric positive
semidefinite $n \times n$ matrices with entries in $I$; denote
$\bp_n(\R)$ by $\bp_n$. We write $A \geq B$ when $A-B \in \bp_n$. For a
function $f: I \to \R$ and a matrix $A \in \bp_n(I)$, we denote by $f[A]$
the matrix $f[A] := (f(a_{ij}))$.
For a matrix $A$ with nonnegative entries, the entrywise power $A^{\circ
\alpha} := ((a_{ij}^\alpha))$ then equals $f_\alpha[A]$. Given a subset
$V \subset \bp_n(I)$, recall \cite{Hiai2009} that a function $f: I \to
\R$ is
\begin{itemize}
\item \emph{positive on $V$} with respect to the Loewner ordering if
$f[A] \geq 0$ for all $0 \leq A \in V$;

\item \emph{monotone on $V$} with respect to the Loewner ordering if
$f[A] \geq f[B]$ for all $A,B \in V$ such that $A \geq B \geq 0$; 

\item \emph{convex on $V$} with respect to the Loewner ordering if
$f[\lambda A + (1-\lambda) B] \leq \lambda f[A] + (1-\lambda) f[B]$ for
all $0 \leq \lambda \leq 1$ and $A, B \in V$ such that $A \geq B \geq 0$. 

\item \emph{super-additive on $V$} with respect to the Loewner
ordering if $f[A+B] \geq f[A] + f[B]$ for all $A,B \in V$ for which
$f[A+B]$ is defined.
\item \emph{sub-additive on $V$} with respect to the Loewner
ordering if $f[A+B] \leq f[A] + f[B]$ for all $A,B \in V$ for which
$f[A+B]$ is defined.
\end{itemize}

\noindent For convenience, functions satisfying these properties are
henceforth termed Loewner positive, Loewner monotone, Loewner convex, and
Loewner super/sub-additive respectively.

Note that many of the critical exponents for Hadamard powers preserving
positivity, monotonicity, and convexity have already been determined in
the literature \cite{FitzHorn, Bhatia-Elsner, Hiai2009}. Yet the sets of
all powers preserving these properties have not been fully characterized.
Specifically, the powers below the critical exponents have not been fully
analyzed. However, when one uses power functions to regularize positive
semidefinite matrices such as correlation matrices, the lower powers are
crucially important, as they produce a lesser degree of perturbation from
the original matrix. So it is also important to classify the powers below
the critical exponent, which preserve positivity. There is thus a
fundamental gap, which is addressed in this paper. Specifically, we
completely characterize all powers $\alpha > 0$ for which the functions
$f_\alpha(x)$, $\phi_\alpha(x)$, and $\psi_\alpha(x)$ are positive,
monotone, convex, or super/sub-additive with respect to the Loewner
ordering, thus completing the analysis.

An important refinement of the above problem is when an additional rank
constraint is imposed. Specifically, we are interested in classifying the
entrywise powers that are Loewner positive, monotone, or convex, when
restricted to matrices in $\bp_n$ of rank at most $k$, for fixed $1 \leq
k \leq n$. Our motivation for imposing such constraints is twofold.
First, for each non-integer power $\alpha < n-2$ below the critical
exponent, one can find low rank $n \times n$ matrices whose positivity is
not preserved by applying $f_\alpha$ entrywise. Preliminary results in
this regard can be found in FitzGerald-Horn \cite{FitzHorn} and
Bhatia-Elsner \cite{Bhatia-Elsner}; however, the role that rank plays in
preserving positivity is not fully understood. It is thus natural to ask
which entrywise powers are Loewner positive, monotone, or convex, when
restricted to positive semidefinite matrices with low rank, or rank
bounded above.
Second, many applications in modern-day high-dimensional statistics
require working with correlation matrices arising from small samples.
Such matrices are very often rank-deficient in practice, and thus it is
useful to characterize maps that preserve positivity when applied to
matrices of a fixed rank.

Before stating our main result, we introduce some notation.

\begin{definition}
Fix integers $n \geq 2$ and $1 \leq k \leq n$, and subsets $I \subset
\R$. Let $\bp_n^k(I)$ denote the subset of matrices in $\bp_n(I)$ that
have rank at most $k$. Define: 
\begin{align}\label{Epos}
\pos(n,k) &:= \{ \alpha \in \R : \mbox{ the function }
x^\alpha \mbox{ is positive on } \bp_n^k([0,\infty)) \}, \notag\\
\pos^\phi(n,k) &:= \{ \alpha \in \R : \mbox{ the function }
\phi_\alpha \mbox{ is positive on } \bp_n^k(\R) \}, \\
\pos^\psi(n,k) &:= \{ \alpha \in \R : \mbox{ the function }
\psi_\alpha \mbox{ is positive on } \bp_n^k(\R) \} \notag.
\end{align}

\noindent Similarly, let $\calh_J(n,k), \calh_J^\phi(n,k),
\calh_J^\psi(n,k)$ denote the entrywise powers preserving Loewner
properties on $\bp_n^k([0,\infty))$ or $\bp_n^k(\R)$, with $J \in
\{$positivity, monotonicity, convexity, super-additivity,
sub-additivity$\}$.
\end{definition}

We now state the main result of this paper in the form of Table
\ref{H-list}.

\begin{theorem}[Main result]\label{thm:main}
Fix an integer $n \geq 2$. The sets of entrywise real powers that are
Loewner positive, monotone, convex, and super/sub-additive, are as listed
in Table \ref{H-list}.
\end{theorem}

\begin{table}
\begin{tabular}{|c|c|c|c|}
\hline
$J$ & $\calh_J(n,k)$ & $\calh_J^\phi(n,k)$ & $\calh_J^\psi(n,k)$\\ \hline
\hline
{\bf Positivity} &\multicolumn{3}{c|}{}\\
\hline
$k=1$ & $\R$ & $\R$ & $\R$\\
& G--K--R & G--K--R & G--K--R \\ \hline
& $\N \cup [n-2,\infty)$ & $2\N \cup [n-2,\infty)$ &
$(-1+2\N) \cup [n-2,\infty)$\\
$2 \leq k \leq n$ & FitzGerald--Horn & FitzGerald--Horn, &
FitzGerald--Horn,\\
& &  Hiai, Bhatia--Elsner, & Hiai, G--K--R\\ & & G--K--R &\\ \hline
{\bf Monotonicity} &\multicolumn{3}{c|}{}\\
\hline
$k=1$ & $[0,\infty)$ & $[0,\infty)$ & $[0,\infty)$\\
& G--K--R & G--K--R & G--K--R \\ \hline
$2 \leq k \leq n$ & $\N \cup [n-1,\infty)$ & $2\N \cup [n-1,\infty)$ &
$(-1+2\N) \cup [n-1,\infty)$\\
& FitzGerald--Horn & FitzGerald--Horn, &
FitzGerald--Horn,\\
& & Hiai, G--K--R & Hiai, G--K--R\\ \hline
{\bf Convexity} &\multicolumn{3}{c|}{}\\
\hline
$k=1$ & $[1,\infty)$ & $[1,\infty)$ & $[1,\infty)$\\
& G--K--R & G--K--R & G--K--R \\ \hline
$2 \leq k \leq n$ & $\N \cup [n,\infty)$ & $2\N \cup [n,\infty)$ &
$(-1+2\N) \cup [n,\infty)$\\
& Hiai, G--K--R & Hiai, G--K--R & Hiai, G--K--R\\ \hline
{\bf {Super-additivity}} &\multicolumn{3}{c|}{}\\
\hline
$1 \leq k \leq n$ & $\N \cup [n,\infty)$ & $2\N \cup [n,\infty)$ &
$(-1+2\N) \cup [n,\infty)$\\
& G--K--R & G--K--R & G--K--R \\ \hline
{\bf {Sub-additivity }}  &\multicolumn{3}{c|}{}\\
\hline
$k=1$ & $(-\infty,0] \cup \{ 1 \} \text{ if } n=2$, & & $\{ 0, 1 \}
\text{ if } n=2$,\\
& $\{ 0,1 \} \text{ if } n>2$ & $\emptyset$ & $\{ 1 \} \text{ if } n>2$\\
& G--K--R & G--K--R & G--K--R \\ \hline
$2 \leq k \leq n$ & $\{ 1 \}$ & $\emptyset$ & $\{ 1 \}$\\
& G--K--R & G--K--R & G--K--R \\ \hline
\end{tabular}
\vspace*{2ex}
\caption{Summary of real Hadamard powers preserving Loewner properties,
with additional rank constraints (G--K--R refers to the current
paper)}\label{H-list}
\end{table}

As the present paper achieves a complete classification of the powers
preserving various Loewner properties, previous contributions in the
literature are also included in Table \ref{H-list} for completeness. Note
that there are many cases which had not been considered previously and
which we settle completely in the paper.
For sake of brevity, we will only briefly sketch proofs for the
previously addressed cases (in order to mention how the rank constraint
affects the problem). We instead focus our attention on the cases that
remain open in the literature. Our original contributions in this paper
are:
\begin{itemize}
\item We complete all of the previously unsolved cases involving powers
preserving positivity, monotonicity, and convexity.
\item We classify all powers preserving super-additivity and
sub-additivity. These properties have not been explored in the
literature in the entrywise setting.
\item We also examine negative powers preserving Loewner properties,
which were also previously unexplored.
\item Finally, we extend all of the above results - as well as those in
the literature - by introducing rank constraints. Once again, we are able
to obtain a complete classification of all real powers preserving the
five aforementioned Loewner properties.
\end{itemize}

Similar to many settings in the literature (see \cite{Walch_survey}), one
can define Hadamard critical exponents for positivity, monotonicity,
convexity, and super-additivity for $\bp_n^k$ - these are the phase
transition points akin to \cite{FitzHorn}. From Theorem \ref{thm:main},
we immediately obtain the Hadamard critical exponents (CE) for the four
Loewner properties for matrices with rank constraints:

\begin{corollary}\label{Ccritexp}
Suppose $n \geq 2$ and $1 \leq k \leq n$. The Hadamard critical exponents
for positivity, monotonicity, convexity, and super-additivity for
$\bp_n^k$ are $n-2$, $n-1$, $n,n$ respectively if $2 \leq k \leq n$, and
$0,0,1,n$ respectively if $k=1$. In particular, they are completely
independent of the type of entrywise power used.
\end{corollary}

An interesting consequence of Corollary \ref{Ccritexp} is that if $k \geq
2$, then the sets of fractional Hadamard powers $f_\alpha, \phi_\alpha$,
or $\psi_\alpha$ that are Loewner positive, monotone, convex, or
super-additive on $\bp_n^k$ do {\it not} depend on $k$. Thus, entrywise
powers that preserve such properties on $\bp_n^2$ automatically preserve
them on all of $\bp_n$.
Corollary \ref{Ccritexp} also shows that the rank $1$ case is different
from that of other $k$, in that three of the critical exponents do not
depend on $n$ if $k=1$. This is not surprising for positivity because the
functions $f_\alpha, \phi_\alpha, \psi_\alpha$ are all multiplicative.
Furthermore, note that if $2 \leq k \leq n$, then entrywise maps are
Loewner convex on $\bp_n^k(I)$ if and only if they are Loewner
super-additive. Finally, the structure of the $\calh_J(n,k)$-sets is
different for $J =$ sub-additivity, compared to the other Loewner
properties.

\subsection*{Organization of the paper}

We prove Theorem \ref{thm:main} by systematically studying entrywise
powers that are (a) positive, (b) monotone, (c) convex, and (d)
super/sub-additive with respect to the Loewner ordering.
Thus in each of the next four sections, we gather previous results from
the literature, and extend these in order to compute the sets
$\calh^I_J(n,k)$ for matrices with rank constraints. In doing so, as a
special case we can complete the classification of powers $f_\alpha,
\phi_\alpha, \psi_\alpha$ that are Loewner positive, monotone, or convex,
for all matrices in $\bp_n = \bp_n^n$ (i.e., with no rank constraint). In
Section \ref{Ssuper}, we then classify the entrywise real powers that are
super/sub-additive and in the process demonstrate an interesting
connection to Loewner convexity. We conclude the paper by discussing
related questions and extensions to other power functions in Section
\ref{S5}.

\section{Characterizing entrywise powers that are Loewner positive}

The study of Hadamard powers originates in the work of FitzGerald and
Horn \cite{FitzHorn}. We begin our analysis by stating one of their main
results that characterizes the Hadamard powers preserving Loewner
positivity. 

\begin{theorem}[FitzGerald and Horn, {\cite[Theorem
2.2]{FitzHorn}}]\label{thm:fitz_horn_fractional}
Suppose $A \in \bp_n([0,\infty))$ for some $n \geq 2$. Then
$A^{\circ\alpha} \in \bp_n$ for all $\alpha \in \N \cup [n-2,\infty)$. If
$\alpha \in (0,n-2)$ is not an integer, then there exists $A \in
\bp_n((0,\infty))$ such that $A^{\circ\alpha} \notin \bp_n$. More
precisely, Loewner positivity is not preserved for $A = ((1 + \epsilon
ij))_{i,j=1}^n$, for all sufficiently small $\epsilon > 0$ with $\alpha
\in (0,n-2) \setminus \N$.
\end{theorem}

\noindent Thus, $\pos(n,n) = \N \cup [n-2,\infty)$ for all $2 \leq n \in
\N$. Additionally, Hiai \cite{Hiai2009} showed that the same results as
above hold for the critical exponent for the even and odd extensions
$\phi_\alpha$ and $\psi_\alpha$:

\begin{theorem}[{Hiai, \cite[Theorem 5.1]{Hiai2009}}]\label{thm:hiai}
If $n \geq 2$ and $\alpha \geq n-2$, then $\alpha \in \pos^\phi(n,n) \cap
\pos^\psi(n,n)$.
\end{theorem}

\begin{remark}
Fix $0 < R \leq \infty$. It is easy to see using a rescaling argument,
that entrywise powers preserving positivity, monotonicity, or convexity
on $\bp_n(-R,R)$ also preserve the respective property on $\bp_n(\R)$,
and vice versa. Thus in the present paper we only work with $\bp_n(\R)$
(or in case of the usual powers $f_\alpha(x) = x^\alpha$, with
$\bp_n([0,\infty))$).
\end{remark}

Theorem \ref{thm:hiai} shows that $[n-2, \infty)$ is contained in both
$\pos^\phi(n,n)$ and $\pos^\psi(n,n)$. It is natural to ask for which
$\alpha \in (0,n-2)$ do the Hadamard powers $\phi_\alpha, \psi_\alpha$
preserve positivity. This question was answered by Bhatia and Elsner for
the powers $\phi_\alpha$:

\begin{theorem}[{Bhatia and Elsner, \cite[Theorem
2]{Bhatia-Elsner}}]\label{thm:bhatia_elsner}
Suppose $n \geq 2$, and $r \in (0,n-2) \setminus 2 \N$ is real. Then $r
\notin \pos^\phi(n,n)$.
\end{theorem}

Note that the above results correspond to the unconstrained-rank case:
$\bp_n(I) = \bp_n^n(I)$. Our main result in this section refines Theorems
\ref{thm:hiai} and \ref{thm:bhatia_elsner}, and completely characterizes
the sets $\pos(n,k)$, $\pos^\phi(n,k)$, and $\pos^\psi(n,k)$ for all $1
\leq k \leq n$.

\begin{theorem}\label{Tpositive}
Suppose $2 \leq k \leq n$ are integers with $n \geq 3$. Then,
\begin{equation}
\pos(n,k) = \N \cup [n-2,\infty), \quad \pos^\phi(n,k) = 2\N \cup
[n-2,\infty), \quad \pos^\psi(n,k) = (-1+2\N) \cup [n-2, \infty).
\end{equation}

\noindent If instead $k=1$ and/or $n=2$, then
\begin{equation}
\pos(n,k) = \pos^\phi(n,k) = \pos^\psi(n,k) = (0,\infty).
\end{equation}
\end{theorem}

\noindent To prove Theorem \ref{Tpositive}, recall the following
classical result about generalized Dirichlet polynomials.

\begin{lemma}[\cite{Laguerre,Bhatia-Elsner}]\label{Ldescartes}
Suppose $\lambda_0 > \lambda_1 > \cdots > \lambda_m > 0$ are real, and $f
: \R \to \R$ is of the form $f(x) := \sum_{i=0}^m a_i \lambda_i^x$ for
some $a_i \in \R$ with $a_0 \neq 0$. Then $f$ has at most $m$ zeros on
the real line.
\end{lemma}

We now proceed to characterize all of the sets $\pos(n,k),
\pos^\phi(n,k), \pos^\psi(n,k)$ for all $1 \leq k \leq n$.

\begin{proof}[Proof of Theorem \ref{Tpositive}]
First suppose $k = 1$, $n \geq 2$, and $A = u u^T \in \bp_n^1$ for some
$u \in \R^n$. Since the functions $f_\alpha, \psi_\alpha, \phi_\alpha$
are multiplicative for all $\alpha \in \R$, we have $A^{\circ \alpha} =
u^{\circ \alpha} (u^{\circ \alpha})^T \in \bp_n^1$ for $u \in
[0,\infty)^n$, and similarly for $\psi_\alpha[A], \phi_\alpha[A]$ for $u
\in \R^n$. The result thus follows for $k=1$. Furthermore, the result is
obvious for $n=2$ and all $\alpha \in \R$.

Now suppose that $2 \leq k \leq n$ and $n \geq 3$. We consider three
cases corresponding to the three functions $f(x) = f_\alpha(x),
\phi_\alpha(x)$, and $\psi_\alpha(x)$. \medskip

\noindent {\bf Case 1: $f(x) = f_\alpha(x)$.}
Consider the matrix
\[ A := \begin{pmatrix} 1 & 1/\sqrt{2} & 0\\ 1/\sqrt{2} & 1 &
1/\sqrt{2}\\ 0 & 1/\sqrt{2} & 1 \end{pmatrix} \oplus {\bf 0}_{(n-3)
\times (n-3)} \in \bp_n^2([0,\infty)). \]

\noindent It is easily verified that $f_\alpha[A] \notin \bp_n$ for all
$\alpha \leq 0$. Thus using Theorem \ref{thm:fitz_horn_fractional}, we
have
\[ \N \cup [n-2, \infty) = \pos(n,n) \subset \pos(n,k) \subset
(0,\infty). \]
Now note that the counterexample $((1 + \epsilon ij)) \in
\bp_n^2([0,\infty))$ provided in Theorem \ref{thm:fitz_horn_fractional}
is a rank $2$ matrix and hence $\alpha \not\in \pos(n,2)$ for any $\alpha
\in (0, n-2) \backslash \N$. Thus $\pos(n,2) = \N \cup [n-2, \infty)$.
Finally, since $\pos(n,k) \subset \pos(n,2)$, it follows that $\pos(n,k)
= \N \cup [n-2, \infty)$. 
\medskip

\noindent {\bf Case 2: $f(x) = \phi_\alpha(x)$.}
Note that $2\N \subset \pos^\phi(n,k)$ by the Schur product theorem.
Using Theorem \ref{thm:hiai} and Case 1, it remains to show that no odd
integer $\alpha \in (0,n-2)$ belongs to $\pos^\phi(n,k)$. To do so, first
define the matrix $A_r$ for $r \in \N$ as follows:
\begin{equation}\label{Erank2}
(A_r)_{ij} := (( \cos(i-j)\pi/r )), \quad 1 \leq i,j \leq r.
\end{equation}

\noindent Note that $A_n \in \bp_n^2$ since $A_n = uu^T +
v v^T$ where $u := (\cos (j \pi/n))_{j=1}^n$ and ${\bf
v} := (\sin (j \pi/n))_{j=1}^n$. By \cite[Theorem 2]{Bhatia-Elsner},  the
matrix $\phi_p[A_{\alpha+3}] \not\in \bp_{\alpha+3}$ for all $p \in
(\alpha-1,\alpha+1)$. In particular, $\phi_\alpha[A_{\alpha+3}] \notin
\bp_{\alpha+3}$. Since we are considering integer powers $\alpha$ such
that $\alpha < n-2$, we have $\alpha+3 \leq n$, so 
\[ A_{\alpha+3} \oplus {\bf 0}_{(n-\alpha-3) \times (n-\alpha-3)} \in
\bp^2_{n}, \qquad \phi_\alpha[ A_{\alpha+3} \oplus {\bf 0}_{(n-\alpha-3)
\times (n-\alpha-3)}] \notin \bp_{n}, \]

\noindent which proves that $\alpha \notin \pos^\phi(n,2)$ for any odd
integer $\alpha \in (0,n-2)$. Since $\pos^\phi(n,k) \subset
\pos^\phi(n,2)$, $\alpha \notin \pos^\phi(n,k)$. This proves Theorem
\ref{Tpositive} for the powers $\phi_\alpha$.\medskip

\noindent {\bf Case 3: $f(x) = \psi_\alpha(x)$.}
Note that $-1+2\N \subset \pos^\psi(n,k)$ by the Schur product theorem.
Using Theorem \ref{thm:hiai} and Case 1, it remains to show that no even
integer $\alpha \in (0,n-2)$ belongs to $\pos^\psi(n,k)$. To do so, we
will construct a matrix $C \in \bp_n^2$ such that $\psi_\alpha[C] \not
\in \bp_n$. We first claim that
\begin{equation}\label{eqn:claim}
\psi_p[A_{\alpha+3}] \notin \bp_{\alpha+3} \quad \forall \ p \in
(\alpha-1,\alpha+1),
\end{equation}

\noindent where $A_{\alpha+3}$ is defined as in \eqref{Erank2}. To prove
the claim, define
\begin{equation}\label{Eevalue}
f(p) := 1 + 2 \sum_{j=1}^{\alpha/2+1} (-1)^j \cos(j \pi / (\alpha+3))^p,
\qquad p \in (0,\infty).
\end{equation}

\noindent Now verify that $(1,-1,1,-1,\dots,1)$ is an eigenvector of
$\psi_p[A_{\alpha+3}]$ with corresponding eigenvalue
\[ \sum_{j=1}^{\alpha+3} (-1)^{j-1} (\psi_p[A_{\alpha+3}])_{1j} = f(p).
\]

\noindent We now show that the eigenvalue $f(p)$ is negative for $p \in
(\alpha-1,\alpha+1)$, proving that $\psi_p[A_{\alpha+3}] \notin
\bp_{\alpha+3}$. The function $f$ satisfies the assumptions of Lemma
\ref{Ldescartes}, and hence $f$ has at most $\alpha/2+1$ zeros on the
real line. Also note that
\[ f(p) = 1 + 2 \sum_{j=1}^{\alpha/2+1} (-1)^j \cos \left( \frac{j
\pi}{\alpha+3} \right)^p = \sum_{j=0}^{\alpha+2} (-1)^j \cos \left(
\frac{j \pi}{n} \right)^p = 0, \qquad \forall p \in (-1 + 2\N) \cap
(0,\alpha+3), \]

\noindent where the last equality is given by \cite[Lemma
2]{Bhatia-Elsner} (to apply {\it loc.~cit.}~we note that $0 < p <
n=\alpha+3$ and $p,n$ are both odd integers, so that $p+n = p + \alpha +
3$ is even).
Therefore $1, 3, \dots, \alpha-1,\alpha+1$ are precisely the $\alpha/2+1$
real roots of $f$. It follows that $f$ is continuous and nonzero on
$(\alpha-1,\alpha+1)$. Moreover, $f'$ has precisely $\alpha/2$ real zeros
$p_1 < \dots < p_{\alpha/2}$ by Rolle's Theorem and Lemma
\ref{Ldescartes}, where
\[ 1 < p_1 < 3 < p_2 < 5 < \dots < \alpha-1 < p_{\alpha/2} < \alpha+1. \]

\noindent Then $f$ is either strictly increasing or strictly decreasing
on $(p_{\alpha/2},\infty)$. Since $f(\alpha+1) = 0$ and $\lim_{p \to
\infty} f(p) = 1$, $f$ must be increasing on $(p_{\alpha/2},\infty)$. In
particular, $f$ is negative on $(\alpha-1,\alpha+1)$. Hence the
eigenvalue $f(p)$ of $\psi_p[A_{\alpha+3}]$ as given in \eqref{Eevalue}
is negative for $p \in (\alpha-1,\alpha+1)$. This proves claim
\eqref{eqn:claim} - in particular, $\psi_\alpha[A_{\alpha+3}] \notin
\bp_{\alpha+3}$. The matrix $C := A_{\alpha+3} \oplus {\bf
0}_{(n-\alpha-3) \times (n-\alpha-3)} \in \bp_n^2$ now satisfies
$\psi_\alpha[C] \not\in \bp_n$, showing that $\alpha \notin
\pos^\psi(n,2)$. It follows that $\alpha \notin \pos^\psi(n,k)$ for all
$2 \leq k \leq n$.
%
\end{proof}

\begin{remark}\label{Rbhatia}
We now come to an open question raised by Bhatia and Elsner in \cite[\S
3]{Bhatia-Elsner} - namely,\smallskip

\noindent {\it Given $p \in (0,\infty) \setminus 2 \N$, find the smallest
$n \in \N$ such that $\phi_p[A] \notin \bp_n$ for at least one matrix $A
\in \bp_n$.}\smallskip

\noindent Our result on the full characterization of the even extensions
of entrywise powers that preserve Loewner positivity, as given by Theorem
\ref{Tpositive}, allows us to answer this question.
By Theorem \ref{Tpositive}, the smallest $n \in \N$ such that $\phi_p[A]
\notin \bp_n$ for at least one matrix $A \in \bp_n$, is $n = \lfloor p
\rfloor + 3$. Similarly, one can ask the analogous question for $\psi$:
{\it given $p \in (0,\infty) \setminus (-1 + 2 \N)$, find the smallest $n
\in \N$ such that $\psi_p[A] \notin \bp_n$ for at least one $A \in
\bp_n$.} Once again by Theorem \ref{Tpositive}, the answer to this
question is $n = \lfloor p \rfloor + 3$.
\end{remark}

\begin{remark}
We note that an alternate approach to proving the claim in
\eqref{eqn:claim} is to conjugate $\psi_{\alpha}[A_{\alpha+3}]$ by the
orthogonal matrix $D_{\alpha+3} =\diag(1, -1, 1, \dots, 1)$ to obtain a
circulant matrix, and then follow the approach in \cite{Bhatia-Elsner}.
\end{remark}

\section{Characterizing entrywise powers that are Loewner monotone}

We now characterize the entrywise powers that are monotone with respect
to the Loewner ordering. The following theorem by FitzGerald and Horn
that is analogous to Theorem \ref{thm:fitz_horn_fractional} but for
monotonicity, answers the question for matrices with nonnegative entries.
In what follows, we denote by ${\bf 1}_{n \times n}$ the $n \times n$
matrix with all entries equal to $1$.

\begin{theorem}[FitzGerald and Horn, {\cite[Theorem
2.4]{FitzHorn}}]\label{thm:HF_monotone}
Suppose $A,B \in \bp_n([0,\infty))$ for some $n \geq 1$. If $A \geq B
\geq 0$, then $A^{\circ \alpha} \geq B^{\circ \alpha} \geq 0$ for $\alpha
\in \N \cup [n-1,\infty)$. If $\alpha \in (0,n-1)$ is not an integer,
then there exist $A \geq B \geq 0$ in $\bp_n([0,\infty))$ such that
$A^{\circ \alpha} \not\geq B^{\circ \alpha}$. More precisely, Loewner
monotonicity is not preserved for $A = ((1 + \epsilon ij))_{i,j=1}^n, B =
{\bf 1}_{n \times n}$, for all sufficiently small $\epsilon > 0$ with
$\alpha \in (0,n-1) \setminus \N$.
\end{theorem}

We now discuss a parallel result to Theorem \ref{thm:HF_monotone} for
$\phi_\alpha$ and $\psi_\alpha$ that was proved by Hiai in
\cite{Hiai2009}. The theorem extends to the cone of positive semidefinite
matrices the standard real analysis result that a differentiable function
is nondecreasing if and only if its derivative is nonnegative.

\begin{theorem}[{Hiai, \cite[Theorems 3.2 and
5.1]{Hiai2009}}]\label{thm:Hiai_monotone}
Suppose $0 < R \leq \infty$, $I = (-R,R)$, and $f : I \to \R$.
\begin{enumerate}
\item For each $n \geq 3$, $f$ is monotone on $\bp_n(I)$ if and only if
$f$ is differentiable on $I$ and $f'$ is Loewner positive on
$\bp_n(I)$.
\item If $n \geq 1$ and $\alpha \geq n-1$, then $\alpha \in
\mono^\phi(n,n) \cap \mono^\psi(n,n)$.
\end{enumerate}
\end{theorem}

\noindent Theorem \ref{thm:Hiai_monotone} is a powerful result, but
cannot be applied directly to study entrywise functions preserving
matrices in the more restricted set $\bp_n^k$. We thus refine the first
part of the theorem to also include rank constraints.

\begin{proposition}\label{prop:mono_implies_pos_rank2}
Fix $0 < R \leq \infty$, $I = (-R, R)$, and $2 \leq k \leq n$.
Suppose $f: I \to \R$ is differentiable on $I$ and Loewner monotone on
$\bp_n^k(I)$. If $A \in \bp_n^k(I)$ is irreducible, then $f'[A] \in
\bp_n$.
\end{proposition}

\begin{proof}
We first make the following observation (which in fact holds over any
infinite field):\medskip

\noindent {\it Suppose $A_{n \times n}$ is a symmetric irreducible
matrix. Then there exists a vector $\zeta \in \img A$ (the image of $A$)
with no zero component.}\medskip

\noindent To see why the observation is true, first suppose that all
vectors in $\img A$ have the $i$th component zero for some $1 \leq i \leq
n$ - i.e., $e_i^T A v = 0$ for all vectors $v$. Then the $i$th row (and
hence column) of $A$ is zero, which contradicts irreducibility. Now fix
vectors $w_i \in \img A$ for all $1 \leq i \leq n$, such that the $i$th
entry of $w_i$ is nonzero. Let $W := [ w_1 | w_2 | \dots | w_n]$; then
for all tuples ${\bf c} := (c_1, \dots, c_n)^T \in \R^n$,
\[ W {\bf c} = \sum_{i=1}^n c_i w_i \in \img A. \]

\noindent Consider the set $S$ of all ${\bf c} \in \R^n$ such that $W
{\bf c}$ has a zero entry. Then $S = \bigcup_{i=1}^n S_i$, where ${\bf c}
\in S_i$ if $e_i^T W {\bf c} = 0$. Note that $S_i$ is a proper subspace
of $\R^n$ since $e_i \notin S_i$ by assumption on $w_i$. Since $\R$ is an
infinite field, $S$ is a proper subset of $\R^n$, which proves the
observation.

Now given an irreducible matrix $A \in \bp_n^k(I)$, choose a vector
$\zeta \in \img A$ as in the above observation. Let $A_\epsilon := A +
\epsilon \zeta \zeta^T$ for $\epsilon > 0$; then $A_\epsilon \in
\bp_n^k(I)$ since $\zeta \in \img A$. Therefore by monotonicity,
$\frac{f[A_\epsilon] - f[A]}{\epsilon} \geq 0$.
Letting $\epsilon \rightarrow 0^+$, it follows that 
$f'[A] \circ (\zeta \zeta^T) \geq 0$. Now let $\zeta^{\circ (-1)} :=
(\zeta_1^{-1}, \dots, \zeta_n^{-1})^T$; then by the Schur Product
Theorem, it follows that
$\displaystyle f'[A] = f'[A] \circ (\zeta \zeta^T) \circ (\zeta^{\circ
(-1)} (\zeta^{\circ (-1)})^T) \geq 0$,
which concludes the proof.
\end{proof}

With the above results in hand, we now completely classify the powers
preserving Loewner monotonicity, and also specify them when rank
constraints are imposed.

\begin{theorem}\label{Tmonotone}
Suppose $2 \leq k \leq n$ are integers. Then,
\begin{equation}
\mono(n,k) = \N \cup [n-1,\infty), \quad \mono^\phi(n,k) = 2\N \cup
[n-1,\infty), \quad \mono^\psi(n,k) = (-1+2\N) \cup [n-1, \infty).
\end{equation}

\noindent If instead $k=1$, then
\begin{equation}
\mono(n,1) = \mono^\phi(n,1) = \mono^\psi(n,1) = (0,\infty).
\end{equation}
\end{theorem}

\begin{proof}
First suppose $k=1<n$ and $A = uu^T, B = vv^T \in \bp_n^1$. If $A \geq B
\geq 0$, then we claim that $v = c u$ for some $c \in [-1,1]$. To see the
claim, assume to the contrary that $u, v$ are linearly independent. We
can then choose $w \in \R^n$ such that $w$ is orthogonal to $u$ but not
to $v$. But then $w^T (A-B) w = - (w^T v)^2 < 0$,
which contradicts the assumption $A \geq B$. Thus $u, v$ are
linearly dependent. Since $A \geq B \geq 0$, it follows that $v = cu$
with $|c| \leq 1$. Now for all $\alpha \geq 0$ and all $A,B \in
\bp_n^1([0,\infty))$ such that $A \geq B \geq 0$, we use the
multiplicativity of $f_\alpha$ to compute:
\[ f_\alpha[A] - f_\alpha[B] = f_\alpha[u] f_\alpha[u]^T -
f_\alpha[cu] f_\alpha[cu]^T = (1 - (c^2)^\alpha)
f_\alpha[u] f_\alpha[u]^T \geq 0. \]

\noindent Thus $f_\alpha$ is monotone on $\bp_n^1([0,\infty))$. Similar
computations show the monotonicity of $\phi_\alpha$ and $\psi_\alpha$ on
$\bp_n^1(\R)$ for all $\alpha \geq 0$. The same computations also show
that $f_\alpha, \phi_\alpha, \psi_\alpha$ are not monotone on
$\bp_n^1(I)$, for any $\alpha < 0$.

Now suppose $2 \leq k \leq n$. Note that if $A \geq B \geq 0$, then one
inductively shows using the Schur product theorem that
\begin{equation}\label{Emonotone}
A^{\circ m} \geq B^{\circ m}\ \forall m \leq N \quad \implies \quad
A^{\circ (N+1)} - B^{\circ (N+1)} = \sum_{m=0}^N A^{\circ m} \circ (A-B)
\circ B^{\circ (N-m)} \geq 0.
\end{equation}

\noindent Therefore every positive integer Hadamard power is monotone on
$\bp_n$. We now consider three cases corresponding to the three functions
$f(x) = f_\alpha(x), \phi_\alpha(x)$, and $\psi_\alpha(x)$. \medskip

\noindent {\bf Case 1: $f(x) = f_\alpha(x)$.}
First suppose $\alpha < 1$. By considering the matrices $B = {\bf 1}_{2
\times 2}$ and $A = B + u u^T$ with $u = (1,-1)^T$, we immediately obtain
that $f_\alpha$ is not monotone on $\bp_2^2([0,\infty))$, and hence not
on $\bp_n^k([0,\infty))$ for all $\alpha < 1$. Now Theorem
\ref{thm:HF_monotone} and the above analysis implies that $\mono(n,k)
\subset \N \cup [n-1,\infty)$, since $((1+\epsilon ij)), {\bf 1}_{n
\times n} \in \bp_n^2([0,\infty)) \subset \bp_n^k([0, \infty))$ provide
the necessary counterexample for $\alpha \in (0,n-1) \setminus \N$.
Furthermore by Theorem \ref{thm:HF_monotone}, $\N \cup [n-1,\infty) =
\mono(n,n) \subset \mono(n,k)$, and thus $\mono(n,k) = \N \cup
[n-1,\infty)$.\medskip

\noindent {\bf Case 2: $f(x) = \phi_\alpha(x)$.}
By Equation \eqref{Emonotone}, $\phi_{2n}[A] = A^{\circ 2n}$ preserves
monotonicity on $\bp_n$. From this observation and Theorem
\ref{thm:Hiai_monotone}, it follows that $2\N \cup [n-1,\infty) \subset
\mono^\phi(n,n) \subset \mono^\phi(n,k)$. We now claim that 
\[ \mono^\phi(n,k) \subset \pos^\phi(n,k) \cap \mono(n,k) \subset \{ n-2
\} \cup 2\N \cup [n-1,\infty). \]

\noindent Indeed, the first inclusion above follows because every
monotone function on $\bp_n^k(\R)$ is simultaneously monotone on
$\bp_n^k([0, \infty))$ and positive on $\bp_n^k(\R)$ by definition. The
second inclusion above holds by Theorem \ref{Tpositive} and Case 1.

It thus remains to consider if $\phi_{n-2}$ is monotone on $\bp_n^k(\R)$.
We consider three sub-cases: if $n>2$ is even, then $n-2 \in 2\N \cup
[n-1,\infty)$. Hence $\mono^\phi(n,k) = 2\N \cup [n-1,\infty)$ by the
analysis mentioned above in Case 2.
Next if $n=3$, we produce a three-parameter family of matrices $A
\geq {\bf 1}_{3 \times 3} \geq 0$ in $\bp_3(\R)$ such that $\phi_1[A]
\not\geq \phi_1[{\bf 1}_{3 \times 3}]$. Indeed, choose any $a>b>0$ and $c
\in (a^{-1}, b^{-1})$, and define
\[ v := (a,b,-c)^T, \qquad B := {\bf 1}_{3 \times 3}, \qquad A := B
+ vv^T. \]

\noindent Then both $A,B$ are in $\bp_3^2(\R)$, and $\phi_1[A], \phi_1[B]
\in \bp_3(\R)$ by Theorem \ref{Tpositive}. However,
\[ \det (\phi_1[A] - \phi_1[B]) = \det \begin{pmatrix}
a^2 & ab & ac-2\\
ab & b^2 & -bc\\
ac-2 & -bc & c^2
\end{pmatrix} = - 4b^2 (ac-1)^2 < 0. \]

\noindent Thus $\phi_1$ is not monotone on $\bp_3^2(\R)$. 

Finally, suppose $n>3$ is odd and that $\phi_{n-2}$ is monotone on
$\bp_n^k(\R)$. We then obtain a contradiction as follows: recall from
Equation \eqref{eqn:claim} that the matrix $A_n$ constructed in Equation
\eqref{Erank2} satisfies: $\psi_{n-3}[A_n] \notin \bp_n$. (Here, $\alpha
= n-3 = p$ is an even integer in $(0,n-2)$, since $n>3$ is odd.)
Moreover, $A_n \in \bp_n^2(\R) \subset \bp_n^k(\R)$ is irreducible.
Hence if $\phi_{n-2}$ is monotone on $\bp_n^k(\R)$, then by Proposition
\ref{prop:mono_implies_pos_rank2}, $\psi_{n-3}[A_n] = \frac{1}{n-2}
(\phi_{n-2})'[A_n] \in \bp_n$. This is a contradiction and so
$\phi_{n-2}$ is not monotone for odd integers $n>3$. This concludes the
classification of the powers $\phi_\alpha$ that preserve Loewner
monotonicity.\medskip

\noindent {\bf Case 3: $f(x) = \psi_\alpha(x)$.}
This case follows similarly to Case 2, and is included for completeness.
By Equation \eqref{Emonotone}, $\psi_{2n+1}[A] = A^{\circ (2n+1)}$
preserves monotonicity on $\bp_n$. From this observation and Theorem
\ref{thm:Hiai_monotone}, it follows that $(-1+2\N) \cup [n-1,\infty)
\subset \mono^\psi(n,n) \subset \mono^\psi(n,k)$. We now claim that 
\[ \mono^\psi(n,k) \subset \pos^\psi(n,k) \cap \mono(n,k) \subset \{ n-2
\} \cup (-1+2\N) \cup [n-1,\infty). \]

\noindent Indeed, the first inclusion above follows because every
monotone function on $\bp_n^k(\R)$ is simultaneously monotone on
$\bp_n^k([0, \infty))$ and positive on $\bp_n^k(\R)$ by definition. The
second inclusion above holds by Theorem \ref{Tpositive} and Case 1.

It thus remains to consider if $\psi_{n-2}$ is monotone on $\bp_n^k(\R)$.
We consider two sub-cases: if $n>2$ is odd, then $n-2 \in (-1+2\N) \cup
[n-1,\infty)$. Hence $\mono^\psi(n,k) = (-1+2\N) \cup [n-1,\infty)$ by
the analysis mentioned above in Case 3. Finally, if $n>2$ is even, then
an argument similar to that for $\phi_{n-2}$ above (together with the
analogous example in Theorem \ref{Tpositive} for $\psi_{n-2}' =
(n-2) \phi_{n-3}$) shows that $n-2 \notin \mono^\psi(n,2)$. In
particular, $n-2 \notin \mono^\psi(n,k)$ for all $2 \leq k \leq n$,
concluding the proof.
\end{proof}

\section{Characterizing entrywise powers that are Loewner convex}

We next characterize the entrywise powers that preserve Loewner
convexity. Before proving the main result of this section, we need a few
preliminary results. Recall that an $n \times n$ matrix $A$ is said to be
\emph{completely positive} if $A =  C C^T$ for some $n \times m$ matrix
$C$ with nonnegative entries. We denote the set of $n \times n$
completely positive matrices by $\cp_n$. 

\begin{lemma}\label{lem:HiaiCP}
Suppose $I \subset \R$ is convex, $n \geq 2$, and $f: I \to \R$ is
continuously differentiable. Given two fixed matrices $A, B \in \bp_n(I)$
such that 
\begin{enumerate}
\item $A - B \in \cp_n$; 
\item $f[\lambda A + (1-\lambda)B] \leq \lambda f[A] + (1-\lambda) f[B]$
for all $0 \leq \lambda \leq 1$.
\end{enumerate}
Then $f'[A] \geq f'[B]$. 
\end{lemma}

\begin{proof}
Since $A - B \in \cp_n$, there exist vectors $v_1, \dots, v_m \in [0,
\infty)^n$ such that 
\begin{equation}
A - B = v_1 v_1^T + \dots + v_m v_m^T. 
\end{equation}
For $1 \leq k \leq m$, let $A_k = B + v_{k+1} v_{k+1}^T + \dots + v_m
v_m^T$. Then $A =: A_0 \geq A_1 \geq \dots \geq A_{m-1} \geq A_m := B$. 
The rest of the proof is the same as the first part of the proof of
\cite[Theorem 3.2(1)]{Hiai2009}.
\end{proof}

Just as Proposition \ref{prop:mono_implies_pos_rank2} was used in proving
Theorem \ref{Tmonotone}, we need the following preliminary result to
classify the powers that preserve convexity.

\begin{proposition}\label{prop:conv_implies_pos_rank2}
Fix $0 < R \leq \infty$, $I = (-R, R)$, and $2 \leq k \leq n$.
Suppose $f: I \to \R$ is twice differentiable on $I$ and Loewner convex
on $\bp_n^k(I)$. If $A \in \bp_n^k(I)$ is irreducible, then $f''[A] \in
\bp_n$.
\end{proposition}

\begin{proof}
Given an irreducible matrix $A \in \bp_n^k(I)$, choose a vector $\zeta
\in \img A$ as in the observation at the beginning of the proof of
Proposition \ref{prop:mono_implies_pos_rank2}.
We now adapt the proof of \cite[Theorem 3.2(1)]{Hiai2009} for the $k=n$
case, to the $2 \leq k < n$ case. Let $A_1 := A + \zeta \zeta^T$; then
$A_1 \in \bp_n^k(I)$ for $\| \zeta \|$ small enough since $\zeta \in \img
A$. More generally, it easily follows that $\lambda A_1 + (1-\lambda) A
\in \bp_n^k(I)$ for all $\lambda \in [0,1]$. Since
\[ f[\lambda A_1 + (1-\lambda) A] \leq \lambda f[A_1] + (1-\lambda) f[A]
\qquad \forall 0 \leq \lambda \leq 1 \]
by convexity, it follows for $0 < \lambda < 1$ that
\begin{align*}
\frac{f[A + \lambda (A_1-A)]-f[A]}{\lambda} &\leq f[A_1] - f[A], \notag\\
\frac{f[A_1 + (1-\lambda)(A-A_1)] - f[A_1]}{1-\lambda} & \leq f[A] -
f[A_1]. 
\end{align*}

\noindent Letting $\lambda \rightarrow 0^+$ or $\lambda \rightarrow 1^-$,
we obtain
\[ (A_1-A) \circ f'[A] \leq f[A_1] - f[A], \qquad
(A-A_1) \circ f'[A_1] \leq f[A] - f[A_1]. \]

\noindent Summing these two inequalities gives $(A_1-A) \circ (f'[A_1] -
f'[A]) \geq 0$. Note that $(A_1 - A)^{\circ -1} = (\zeta \zeta^T)^{\circ
-1} \in \bp_n^1$ and so $f'[A_1] - f'[A] \geq 0$.

Finally, given $\epsilon > 0$, define $A_\epsilon := A + \epsilon \zeta
\zeta^T$. Then $A_\epsilon \in \bp_n^k(I)$ and $f'[A_\epsilon] \geq
f'[A]$ by the previous paragraph for $\sqrt{\epsilon} \zeta$. Therefore,
for all $\epsilon > 0$,
$\frac{f'[A_\epsilon] - f'[A]}{\epsilon} \geq 0$.
Letting $\epsilon \rightarrow 0^+$, it follows that 
$f''[A] \circ (\zeta \zeta^T) \geq 0$. Now let $\zeta^{\circ (-1)} :=
(\zeta_1^{-1}, \dots, \zeta_n^{-1})^T$; then by the Schur Product
Theorem,
\[ 0 \leq f''[A] \circ (\zeta \zeta^T) \circ (\zeta^{\circ (-1)}
(\zeta^{\circ (-1)})^T) = f''[A], \]

\noindent which concludes the proof.
\end{proof}

Note that Lemma \ref{lem:HiaiCP} and Proposition
\ref{prop:conv_implies_pos_rank2} generalize to the cone $\bp_n$ of
matrices with the Loewner ordering, the elementary results from real
analysis that the first and second derivatives of a convex (twice)
differentiable function are nondecreasing and nonnegative, respectively.
These parallels have been explored by Hiai in detail for $I = (-R,R)$;
see \cite[Theorems 3.2 and 5.1]{Hiai2009}. We now state some assertions
from {\it loc.~cit.}~that concern Loewner convexity.

\begin{theorem}[{Hiai, \cite[Theorems 3.2 and
5.1]{Hiai2009}}]\label{thm:Hiai_convex}
Suppose $0 < R \leq \infty$, $I = (-R,R)$, and $f : I \to \R$.
\begin{enumerate}
\item For each $n \geq 2$, $f$ is convex on $\bp_n(I)$ if and only
if $f$ is differentiable on $I$ and $f'$ is monotone on $\bp_n(I)$.
\item If $n \geq 1$ and $\alpha \geq n$, then $\alpha \in \conv^\phi(n,n)
\cap \conv^\psi(n,n)$.
\end{enumerate}
\end{theorem}

With the above results in hand, we now extend them in order to completely
classify the powers preserving Loewner convexity, and also specify them
when rank constraints are imposed.

\begin{theorem}\label{Tconvex}
Suppose $2 \leq k \leq n$ are integers. Then,
\begin{equation}
\conv(n,k) = \N \cup [n,\infty), \quad \conv^\phi(n,k) = 2\N \cup
[n,\infty), \quad \conv^\psi(n,k) = (-1+2\N) \cup [n, \infty).
\end{equation}

\noindent If instead $k=1$, then
\begin{equation}
\conv(n,1) = \conv^\phi(n,1) = \conv^\psi(n,1) = [1,\infty).
\end{equation}
\end{theorem}

\begin{proof}
Suppose that $k=1<n$ and $A = u u^T, B = v v^T
\in \bp_n^1$. If $A \geq B \geq 0$, then by the proof of Theorem
\ref{Tmonotone}, $v = c u$ for some $c \in [-1,1]$. Now for
any $\alpha > 0$ and $\lambda \in [0,1]$,
\[ \lambda f_\alpha[A] + (1-\lambda) f_\alpha[B] - f_\alpha[\lambda A +
(1 - \lambda) B] = (\lambda + (1-\lambda) c^{2\alpha} - (\lambda +
c^2(1-\lambda))^\alpha) f_\alpha[A]. \]

\noindent So $f_\alpha$ is convex on $\bp_n^1(\R)$ if and only if (using
$b = c^2$)
\[ \lambda + (1-\lambda) b^\alpha \geq (\lambda + b(1-\lambda))^\alpha,
\qquad \forall \lambda,b \in [0,1]. \]

\noindent This condition is equivalent to the function $x \mapsto
x^\alpha$ being convex on $[0,1]$ and hence on $[0,\infty)$ - in other
words, if and only if $\alpha \geq 1$. A similar argument can be applied
to analyze $\phi_\alpha, \psi_\alpha$.
If on the other hand $\alpha < 1$, then set $A := {\bf 1}_{2 \times 2}
\oplus {\bf 0}_{(n-2) \times (n-2)} \in \bp_n^1(I)$ and $B := {\bf 0}_{n
\times n}$, and compute:
\begin{equation}\label{Enotconvex}
\frac{1}{2} f[A] + \frac{1}{2} f[B] - f[\frac{1}{2}(A + B)] =
\frac{1}{2} f[A] - f[\frac{1}{2}A] = (2^{-1} - 2^{-\alpha}) f[A],
\end{equation}

\noindent which is clearly not in $\bp_n(\R)$ if $\alpha < 1$. It follows
that none of the functions $f = f_\alpha, \phi_\alpha, \psi_\alpha$ is
convex on $\bp_n^k(I)$ for $\alpha < 1$, $n \geq 2$, and $1 \leq k \leq
n$.

We now assume that $2 \leq k \leq n$, and show that $\N \cup [n,\infty)
\subset \conv(n,k)$. We first assert that for any differentiable function
$f: [0, \infty) \to \R$ such that $f'(x)$ is monotone on
$\bp_n([0,\infty))$, then $f$ is convex on $\bp_n([0,\infty))$. This
assertion parallels one implication in Theorem \ref{thm:Hiai_convex}(1)
for $I = [0,\infty)$ instead of $I = (-R,R)$. As the proof is similar to
the proof of \cite[Proposition 2.4]{Hiai2009}, it is omitted. 

Next, letting $f(x) = x^\alpha$ for $\alpha \in [n, \infty) \cup \N$, it
follows immediately from Theorem \ref{Tmonotone} that $f$ is convex on
$\bp_n([0,\infty))$. Thus $\N \cup [n, \infty) \subset \conv(n,n) \subset
\conv(n,k)$. Now note that for any $\alpha \geq 1$, $\phi_\alpha'(x) =
\alpha \psi_{\alpha-1}(x)$ and $\psi_\alpha'(x) =
\alpha\phi_{\alpha-1}(x)$.
Thus using Theorem \ref{thm:Hiai_convex}, it follows that $2\N \cup [n,
\infty) \subset \conv^\phi(n,k)$ and $(-1+2\N) \cup [n, \infty) \subset
\conv^\psi(n,k)$.

Note also from Equation \eqref{Enotconvex} that $\conv(n,k) \subset
[1,\infty)$, and similarly for $\conv^\phi(n,k)$ and $\conv^\psi(n,k)$.
Thus to show the reverse inclusions - i.e., that $\conv(n,k) \subset \N
\cup [n, \infty)$ (and analogously for $\phi_\alpha, \psi_\alpha$), we
consider three cases corresponding to the three functions $f = f_\alpha,
\phi_\alpha, \psi_\alpha$.\medskip

\noindent {\bf Case 1: $f(x) = f_\alpha(x)$.}
Let $\alpha \in \conv(n,k)$ and consider the matrices $A = A_\epsilon =
((1+\epsilon ij))_{i,j=1}^n$ and $B = {\bf 1}_{n \times n}$ for $\epsilon
> 0$. Since $A -B \in \cp_n$, by Lemma \ref{lem:HiaiCP} for $I = [1,
1+\epsilon n^2]$, we have $f'_\alpha[A] \geq f'_\alpha[B]$. Thus by
Theorem \ref{thm:HF_monotone}, it follows that $\alpha - 1 \geq n-1$ or
$\alpha \in \N$. Therefore $\conv(n,k) \subset \N \cup [n, \infty)$.
\medskip

\noindent {\bf Case 2: $f(x) = \phi_\alpha(x)$.}
Given $\alpha \in \conv^\phi(n,k)$ for $k \geq 2$, first note by Case 1
that
\begin{equation}\label{Econvex}
\conv^\phi(n,k)  \subset \conv^\phi(n,2) \subset \conv(n,2) = \N \cup
[n,\infty).
\end{equation}

\noindent Thus it suffices to show that there is no odd integer in $S :=
(0,n) \cap \conv^\phi(n,2)$. First note that for every odd integer
$\alpha \in S$, the function $\phi_\alpha$ is convex on
$\bp_{\alpha+1}^2(\R)$. There are now two cases: first if $\alpha>1$,
then define $A_{\alpha+1} \in \bp_{\alpha+1}^2(\R)$ as in \eqref{Erank2}.
Now $A_{\alpha+1}$ is irreducible since $\alpha \geq 3$. Applying
Proposition \ref{prop:conv_implies_pos_rank2} to $A_{\alpha+1}$, we
obtain $\phi_\alpha''[A_{\alpha+1}] \geq 0$. Now if $2 \leq \alpha < n$,
then this contradicts Case 2 of the proof of Theorem \ref{Tpositive}
since $\alpha$ is an odd integer. Therefore $\alpha \notin
\conv^\phi(n,2)$ for all odd integers $\alpha \in (1,n)$. The second case
is when $\alpha = 1$. Recall from \cite[Proposition 2.4]{Hiai2009} that
if $\alpha =1$ and $\phi_1 : \R \to \R$ is convex on
$\bp_{\alpha+1}^2(\R) = \bp_2(\R)$, then $\phi_1(x) = |x|$ would be
differentiable on $\R$, which is false. We conclude that $\conv^\phi(n,2)
\subset 2\N \cup [n,\infty)$.\medskip

\noindent {\bf Case 3: $f(x) = \psi_\alpha(x)$.}
We now prove that $\conv^\psi(n,2) \subset (-1+2\N) \cup
[n,\infty)$. Once again it suffices to show that no even integer $\alpha
\in (0,n)$ lies in $\conv^\psi(n,2)$. First assume that $\alpha>2$. Then
an argument similar to that for $\phi_\alpha$ above (together with the
analogous example in Theorem \ref{Tpositive} for $\psi_\alpha$) shows
that $\alpha \notin \conv^\psi(n,2)$. Finally, if $\alpha=2$, we provide
a three-parameter family of counterexamples to show that $\psi_2$ is not
convex on $\bp_3^2(\R)$ (and hence not convex on $\bp_n^2(\R)$ by adding
blocks of zeros). To do so, choose $0 < b < a < \infty$ and $c \in
(a^{-1}, \min(b^{-1}, 2a^{-1}))$, and define:
\[ v := (a,b,-c)^T, \qquad B := {\bf 1}_{3 \times 3}, \qquad A := B
+ v v^T. \]

\noindent Clearly, $A, B \in \bp^2_3(\R)$ and  $A \geq B$. Moreover,  
\[ C:= \frac{1}{2} (\psi_2[A] + \psi_2[B]) - \psi_2[ (A+B)/2] =
\frac{1}{4} \begin{pmatrix} a^4 & a^2 b^2 & (3ac-2)(2-ac)\\ a^2 b^2 & b^4
& b^2 c^2\\ (3ac-2)(2-ac) & b^2 c^2 & c^4 \end{pmatrix}. \]

\noindent Now verify that $\det C = - 4^{-3} [2b(ac-1)]^4 < 0$. Thus $A
\geq B \geq 0$ provide a family of counterexamples to the convexity of
$\psi_2$ on $\bp_3^2(\R)$ (with $\lambda = 1/2$).
%
%
\end{proof}

\begin{remark}
In \cite[Lemma 5.2]{Hiai2009}, Hiai generalizes the arguments of
FitzGerald and Horn in \cite[Theorems 2.2, 2.4]{FitzHorn} and proves that
for a given non-integer $0 < \alpha < n$, if $A_\epsilon := ((1+\epsilon
ij))_{i,j=1}^n$, then  
\[ f_\alpha\left[\frac{A_\epsilon+A_{\epsilon'}}{2}\right] \not\leq
\frac{f_\alpha[A_\epsilon]+f_\alpha[A_{\epsilon'}]}{2} \]

\noindent for some $\epsilon, \epsilon' \geq 0$ small enough. Hiai's
result can be used to deduce that $\conv(n,k) \subset \N \cup [n,
\infty)$ for $2 \leq k < n$. The proof of Theorem \ref{Tconvex} above is
different in that it builds directly on the results obtained for
monotonicity as compared to constructing specific matrices. The proof
also has the additional advantage that it classifies the integer powers
$x^m$ for $m < n$, which are Loewner convex.
\end{remark}

\section{Characterizing entrywise powers that are Loewner
super/sub-additive}\label{Ssuper}

Powers that are Loewner super/sub-additive have been studied for matrix
functions in parallel settings, where functions of matrices are evaluated
through the Hermitian functional calculus instead of entrywise (see e.g.
\cite{Ando88, AndoZhan99, SinghVasudeva2002, BourinUchiyama2007,
Bourin2010, AudenaertAujla2012}). We now characterize the powers that are
Loewner super/sub-additive when applied entrywise. 

\begin{theorem}\label{Tsuper}
Suppose $1 \leq k \leq n$ are integers with $n \geq 2$. Then,
\begin{enumerate}
\item $\super(n,k) = \N \cup [n,\infty), \quad \super^\phi(n,k) = 2\N
\cup [n,\infty), \quad \super^\psi(n,k) = (-1+2\N) \cup [n, \infty).$

\item
\begin{enumerate}
\item
$\displaystyle \sub(n,k) = \begin{cases}
\{ 1 \}, & \text{ if } 2 \leq k \leq n,\\
\{ 0, 1 \}, & \text{ if } k=1,n>2,\\
(-\infty,0] \cup \{ 1 \}, & \text{ if } (n,k) = (2,1).
\end{cases}$

\item $\sub^\phi(n,k) = \emptyset$ for all $1 \leq k \leq n$.

\item $\sub^\psi(n,k) = \{ 0, 1 \}$ if $(n,k) = (2,1)$, and $\{ 1 \}$
otherwise.
\end{enumerate}
\end{enumerate}
\end{theorem}

Before we prove the result, note that it yields a 
hitherto unknown connection between super-additivity and convexity with
respect to the Loewner ordering.

\begin{corollary}\label{Csuper}
Fix $\alpha > 0$ and integers $2 \leq k \leq n$. A fractional power
function $f = f_\alpha, \phi_\alpha, \psi_\alpha$ is Loewner convex on
$\bp_n^k(I)$ if and only if $f$ is Loewner super-additive. Here $I =
[0,\infty)$ if $f = f_\alpha$ and $I=\R$ otherwise. 
\end{corollary}

\begin{proof}
The result follows from Theorems \ref{Tconvex} and \ref{Tsuper}.
\end{proof}

\begin{remark}\label{Rsuper}
Note {\it a priori} that the defining inequalities for convexity and
super-additivity go in opposite ways. More precisely, for $\alpha > 0$
Loewner convexity is equivalent to Loewner midpoint convexity by
continuity. Thus, $f = f_\alpha, \phi_\alpha, \psi_\alpha$ is convex if
and only if $f[A+B] \leq 2^{\alpha-1} (f[A] + f[B])$ for $A \geq B \geq
0$. On the other hand, super-additivity asserts that $f[A+B] \geq f[A] +
f[B]$ for $A,B \geq 0$. However, recall that a convex function $f :
[0,\infty) \to [0,\infty)$ is super-additive if and only if $f(0) = 0$.
The characterization of entrywise Loewner convex powers in Corollary
\ref{Csuper} in terms of the super-additive ones is thus akin to the
aforementioned fact for $n=1$.
\end{remark}

In order to prove Theorem \ref{Tsuper}, we extend classical results about
generalized Vandermonde determinants to the odd and even extensions of
the power functions. 

\begin{proposition}\label{Pvandermonde}
Let $0 < R \leq \infty$. Then,
\begin{enumerate}
\item the functions $\{ f_\alpha : \alpha \in \R \} \cup \{ f \equiv 1
\}$ are linearly independent on $I = [0,R)$.
\item the functions $\{ \phi_\alpha, \psi_\alpha : \alpha \in \R \} \cup
\{ f \equiv 1 \}$ are linearly independent on $I = (-R,R)$.
\end{enumerate}
\end{proposition}

\begin{proof}
Fix $\alpha_1 < \cdots < \alpha_n$ and define $\alpha := (\alpha_1,
\dots, \alpha_n)$. We first show that the set of functions
$\{x^{\alpha_i}: i=1,\dots,n\} \cup \{f \equiv 1\}$ is linearly
independent on $[0,R)$. Indeed, fix ${\bf x} := (0, x_1, \dots, x_n) \in
\R^n$ for any $0 < x_1 < \cdots < x_n < R$; then by \cite[Chapter XIII,
\S8, Example 1]{Gantmacher_Vol2}, the vectors
\[ {\bf x}^{\circ \alpha_j} := \left(0, x_1^{\alpha_j}, \dots,
x_n^{\alpha_j} \right), \qquad j=1, \dots, n \]

\noindent and $(1, 1, \dots, 1)$ are linearly independent. 

We next show that the set of functions $\{\phi_{\alpha_i},
\psi_{\alpha_i}: i=1,\dots,n\} \cup \{f \equiv 1\}$ is linearly
independent on $(-R,R)$. Indeed, fix ${\bf x'} := (x_1, \dots, x_n)$ with
$x_i \in (0,R)$ as above; then by the above analysis,
\begin{equation}\label{EPsi}
\Psi({\bf x'},{\bf \alpha}) := \begin{pmatrix}
(\phi_{\alpha_i}(x_j))_{i,j=1}^n & (\phi_{\alpha_i}(-x_j))_{i,j=1}^n \\
(\psi_{\alpha_i}(x_j))_{i,j=1}^n & (\psi_{\alpha_i}(-x_j))_{i,j=1}^n
\end{pmatrix}
\end{equation}

\noindent is a nonsingular matrix, since it is of the form
$\displaystyle \begin{pmatrix}M & M \\ M & -M\end{pmatrix}$ for a
nonsingular matrix $M$. But then
$\displaystyle \begin{pmatrix} \Psi({\bf x'}, \alpha) & {\bf 0}_{2n
\times 1}\\ {\bf 1}_{1 \times 2n} & 1 \end{pmatrix}$
is also nonsingular, whence the points $\pm x_1, \dots, \pm x_n, 0$
provide the required nonsingular matrix. This proves the second
assertion.
\end{proof}

Proposition \ref{Pvandermonde} has the following consequence that is
repeatedly used in proving Theorem \ref{Tsuper}. 

\begin{corollary}\label{Cindependent}
Let $0 < R \leq \infty$ and $I = (-R,R)$ or $I = [0,R)$. Fix integers $1
\leq m \leq n$ and scalars $c_1, \dots, c_m$ and $\alpha_1 < \alpha_2 <
\cdots < \alpha_m$. Suppose $\{ g_1, \dots, g_m \} \subset \{
\phi_{\alpha_1}, \dots, \phi_{\alpha_m}, \psi_{\alpha_1}, \dots,
\psi_{\alpha_m} \}$, and define $F(x) := \sum_{i=1}^m c_i g_i(x)$. Then
there exist vectors $u \in (I \cap (-1,1))^n$ and $v_i \in \R^n$ that do
not depend on $c_i$, such that $v_i^T F[u u^T] v_i = c_i$ for all
$i=1,\dots,m$.
\end{corollary}

\begin{proof}
Suppose first $I = (-R,R)$. Choose scalars $\alpha_n > \alpha_{n-1} >
\cdots > \alpha_{m+1} > \alpha_m$. By Proposition \ref{Pvandermonde}, the
functions $\phi_{\alpha_1}, \dots, \phi_{\alpha_n}, \psi_{\alpha_1},
\dots, \psi_{\alpha_n}$ are linearly independent. Thus, as in the proof
of Proposition \ref{Pvandermonde}, for all pairwise distinct $x_1, \dots,
x_n \in (0,1) \cap I$, the matrix $\Psi({\bf x}, \alpha)$ as in
\eqref{EPsi} is nonsingular, where ${\bf x} := (x_1, \dots, x_n)$ and
$\alpha :=(\alpha_1, \dots, \alpha_n)$. Now consider the submatrix $C_{m
\times 2n}$ of $\Psi({\bf x}, \alpha)$ whose rows correspond to the
functions $g_i$ for $1 \leq i \leq m$. Since $C$ has full rank, choose
elements $u_1, u_2, \dots, u_n$ from among the $\pm x_i$ such that the
matrix $(g_i(u_j))_{i,j=1}^m$ is nonsingular. Now set $u := (u_1, \dots,
u_n)^T$; then the vectors $g_i[u]$ are linearly independent. Choose $v_i$
to be orthogonal to $g_j[u]$ for $j \neq i$, and such that $v_i^T g_i[u]
= 1$. It follows that $v_i^T F[u u^T] v_i = c_i$ for all $i$. The proof
is similar for $I = [0,R)$.
\end{proof}

We now classify the entrywise powers that are Loewner super/sub-additive.

\begin{proof}[Proof of Theorem \ref{Tsuper}]\hfill

\noindent {\bf (1) Super-additivity.}
Fix an integer $1 \leq k \leq n$. First apply the definition of
super-additivity to $A = B = {\bf 1}_{n \times n} \in
\bp_n^k([0,\infty))$ to conclude that if $\alpha \in \R$ and one of
$f_\alpha, \phi_\alpha, \psi_\alpha$ is Loewner super-additive, then
$\alpha \geq 1$. We now consider three cases corresponding to the three
functions $f(x) = f_\alpha(x), \phi_\alpha(x)$, and $\psi_\alpha(x)$ for
$\alpha \geq 1$.\medskip

\noindent {\bf Case 1: $f(x) = f_\alpha(x)$.}
That $f_\alpha$ is super-additive on $\bp_n([0,\infty))$ for $\alpha \in
\N$ follows by applying the binomial theorem. Now, suppose $\alpha \in
(1,\infty) \setminus \N$. We adapt the argument in \cite[Theorem
2.4]{FitzHorn} to our situation. First assume that $\alpha \geq n$; then
for $A,B \in \bp_n([0,\infty))$, 
\[ f_\alpha[A+B] = f_\alpha[A] + \alpha \int_0^1 B \circ
f_{\alpha-1}[t(A+B)+(1-t)A]\ dt. \]

\noindent Note that $t(A+B) + (1-t)A = A+tB \geq tB$ for all $0 \leq t
\leq 1$. Since $\alpha-1 \geq n-1$, it follows by Theorem
\ref{thm:HF_monotone} that $f_{\alpha-1}[t(A+B) + (1-t)A] \geq
t^{\alpha-1} f_{\alpha-1}[B]$. Therefore,
\[ f_\alpha[A+B] \geq f_\alpha[A] + \alpha f_\alpha[B] \int_0^1
t^{\alpha-1}\ dt = f_\alpha[A]+f_\alpha[B]. \]

\noindent This shows that $f_\alpha$ is super-additive on
$\bp_n([0,\infty))$, and hence on $\bp_n^k([0,\infty))$ if $\alpha \geq
n$. The last remaining case is when $\alpha \in (1,n) \setminus \N$.
Define $g_\alpha(x) := (1+x)^\alpha$. Given $\epsilon > 0$ and $v \in
(0,1)^n$, apply Taylor's theorem entrywise to $g_\alpha[\epsilon v v^T]$
to obtain:
\begin{equation}\label{eqn:taylor_g}
g_\alpha[\epsilon v v^T] = {\bf 1}_{n \times n} + \sum_{i=1}^{\lfloor
\alpha \rfloor} \epsilon^i \binom{\alpha}{i} f_i[v] f_i[v]^T +
O(\epsilon^{1+ \lfloor \alpha \rfloor}) C, 
\end{equation}

\noindent where $C = C(v)$ is an $n \times n$ matrix that is independent
of $\epsilon$. By Corollary \ref{Cindependent} applied to $F(x) =
\sum_{i=1}^{\lfloor \alpha \rfloor} \epsilon^i \binom{\alpha}{i} x^i -
\epsilon^\alpha x^\alpha$ and $m = 1 + \lfloor \alpha \rfloor \leq n$,
there exist $u \in (0,1)^n$ and $x_\alpha \in \R^n$ such that $x_\alpha^T
F[u u^T] x_\alpha = -\epsilon^\alpha$. It follows that
\[ x_\alpha^T \left(f_\alpha[{\bf 1}_{n \times n} + \epsilon u u^T] -
{\bf 1}_{n \times n} - \epsilon^\alpha f_\alpha[u u^T]\right) x_\alpha =
O(\epsilon^{1+ \lfloor \alpha \rfloor}) x_\alpha^T C x_\alpha -
\epsilon^\alpha, \]

\noindent and the last expression is negative for sufficiently small
$\epsilon = \epsilon_0 > 0$. Hence $f_\alpha[{\bf 1}_{n \times n} +
\epsilon_0 u u^T] \not\geq f_\alpha[{\bf 1}_{n \times n}] +
f_\alpha[\epsilon_0 u u^T]$. This shows that $f_\alpha$ is not
super-additive on $\bp_n^1([0,\infty))$ and hence on
$\bp_n^k([0,\infty))$, for $\alpha \in (1,n) \setminus \N$.\medskip

\noindent {\bf Case 2: $f(x) = \phi_\alpha(x)$.}
Clearly, the assertion holds if $\alpha \in 2\N$ and $1 \leq k \leq n$,
since in that case $\phi_\alpha \equiv x^\alpha$. Next if $\alpha \geq n
\geq 2$, then as in Case 1, for $A,B \in \bp_n(\R)$, 
\[ \phi_\alpha[A+B] = \phi_\alpha[A] + \alpha \int_0^1 B \circ
\psi_{\alpha-1}[t(A+B)+(1-t)A]\ dt. \]

\noindent Since $\alpha-1 \geq n-1$, by Theorem \ref{Tmonotone} the
function $\psi_{\alpha-1}$ is monotone on $\bp_n(\R)$. Thus,
\[ \phi_\alpha[A+B] \geq \phi_\alpha[A] + \alpha B \circ
\psi_{\alpha-1}[B] \int_0^1 t^{\alpha-1}\ dt =
\phi_\alpha[A]+\phi_\alpha[B]. \]

\noindent It follows that $\phi_\alpha$ is super-additive on
$\bp_n^k(\R)$ for $\alpha \in 2 \N \cup [n,\infty)$. Next note by Case
(1) that $\phi_\alpha$ is not super-additive on $\bp_n^k(\R)$ for $\alpha
\in (1,n) \setminus \N$. It thus remains to prove that $\phi_\alpha$ is
not super-additive on $\bp_n^k(\R)$ for $\alpha \in (-1+2\N) \cap [1,n)$.
Note that for all $u,v \in \R^n$, if $\phi_\alpha$ is super-additive,
then
\[ \phi_\alpha[uu^T + vv^T] \geq \phi_\alpha[uu^T] + \phi_\alpha[vv^T] =
\phi_\alpha[u]\phi_\alpha[u]^T + \phi_\alpha[v]\phi_\alpha[v]^T \in
\bp_n(\R). \]

\noindent Thus, if $\phi_\alpha$ is super-additive, then it is also
positive on $\bp_n^2(\R)$. We conclude by Theorem \ref{Tpositive} that
$\phi_\alpha$ is not super-additive for $\alpha \in (-1+2\N) \cap
[1,n-2)$. 

The only two powers left to consider are $\alpha = n-2$ with $n$ odd, and
$\alpha = n-1$ with $n$ even. In other words, $n$ is of the form $n = 2l$
or $n=2l+1$ with $l \geq 1$. Thus, $\alpha = 2l-1 \geq 1$ in both cases.
We claim that $\phi_{2l-1}$ is not super-additive on $\bp_n^1(\R)$. To
show the claim, first observe that if $v \in (-1,1)^n$, then $1+v_iv_j >
0$ for all $i,j$, and so by the binomial theorem, 
\[ \phi_{2l-1}[{\bf 1}_{n \times n} + v v^T] - {\bf 1}_{n \times n} -
\phi_{2l-1}[v v^T] = \sum_{i=1}^{2l-1} \binom{2l-1}{i} g_i[v v^T] -
\phi_{2l-1}[v v^T], \]

\noindent where $g_i(x) = x^i$ for $x \in \R$ and $i = 1, \dots, 2l-1$.
By Corollary \ref{Cindependent} applied to the functions $g_1, g_2,
\dots, g_{2l-1} = \psi_{2l-1}, \phi_{2l-1}$ and $m = 2l \leq n$, there
exist $u \in (-1,1)^n$ and $x \in \R^n$ such that
\[ x^T \left( \phi_{2l-1}[{\bf 1}_{n \times n} + u u^T] - {\bf 1}_{n
\times n} - \phi_{2l-1}[u u^T] \right) x = -1. \]

\noindent This shows that $\phi_{2l-1}$ is not super-additive on
$\bp_n^1(\R)$, hence not on $\bp_n^k(\R)$.\medskip


\noindent {\bf Case 3: $f(x) = \psi_\alpha(x)$.}
The proof is similar to that of Case 2, and is included for completeness.
Clearly, the assertion holds if $\alpha \in -1+2\N$ and $1 \leq k \leq
n$, since in that case $\psi_\alpha \equiv x^\alpha$. Next if $\alpha
\geq n \geq 2$, then as in Case 1, for $A,B \in \bp_n(\R)$, 
\[ \psi_\alpha[A+B] = \psi_\alpha[A] + \alpha \int_0^1 B \circ
\phi_{\alpha-1}[t(A+B)+(1-t)A]\ dt. \]

\noindent Since $\alpha-1 \geq n-1$, by Theorem \ref{Tmonotone} the
function $\phi_{\alpha-1}$ is monotone on $\bp_n(\R)$. Thus,
\[ \psi_\alpha[A+B] \geq \psi_\alpha[A] + \alpha B \circ
\phi_{\alpha-1}[B] \int_0^1 t^{\alpha-1}\ dt =
\psi_\alpha[A]+\psi_\alpha[B]. \]

\noindent It follows that $\psi_\alpha$ is super-additive on
$\bp_n^k(\R)$ for $\alpha \in (-1+2 \N) \cup [n,\infty)$. Next note by
Case 1 that $\psi_\alpha$ is not super-additive on $\bp_n^k(\R)$ for
$\alpha \in (1,n) \setminus \N$. It thus remains to prove that
$\psi_\alpha$ is not super-additive on $\bp_n^k(\R)$ for $\alpha \in 2\N
\cap [1,n)$. As in Case 2, we observe that if $\psi_\alpha$ is
super-additive, then it is also positive on $\bp_n^2(\R)$. Hence by
Theorem \ref{Tpositive}, $\psi_\alpha$ is not super-additive for $\alpha
\in 2\N \cap [1,n-2)$.

The only two powers left to consider are $\alpha = n-2$ with $n$ even,
and $\alpha = n-1$ with $n$ odd. Let $n = 2l+1$ or $2l+2$ with $l \geq
1$, so that $\alpha = 2l \geq 2$ in both cases. We claim that $\psi_{2l}$
is not super-additive on $\bp_n^1(\R)$. To show the claim, first observe
that if $v \in (-1,1)^n$, then $1+v_iv_j > 0$ for all $i,j$, and so by
the binomial theorem, 
\[ \psi_{2l}[{\bf 1}_{n \times n} + v v^T] - {\bf 1}_{n \times n} -
\psi_{2l}[v v^T] = \sum_{i=1}^{2l} \binom{2l}{i} g_i[v v^T] - \psi_{2l}[v
v^T], \]

\noindent where $g_i(x) = x^i$ for $x \in \R$ and $i = 1, \dots, 2l$.
By Corollary \ref{Cindependent} applied to the functions $g_1, g_2,
\dots, g_{2l} = \phi_{2l}, \psi_{2l}$ and $m = 2l+1 \leq n$, there exist
$u \in (-1,1)^n$ and $x \in \R^n$ such that
\[ x^T \left( \psi_{2l}[{\bf 1}_{n \times n} + u u^T] - {\bf 1}_{n \times
n} - \psi_{2l}[u u^T] \right) x = -1. \]

\noindent This shows that $\psi_{2l}$ is not super-additive on
$\bp_n^1(\R)$, hence not on $\bp_n^k(\R)$.
\medskip

\noindent {\bf (2)(a) Sub-additivity for $f_\alpha$.}
First note that if $\alpha \in \R$, applying the definition of
sub-additivity to $A = B = {\bf 1}_{n \times n} \in \bp_n^1([0,\infty))$
shows that $f_\alpha$ is not Loewner sub-additive for $\alpha > 1$.
Clearly $f_1$ is sub-additive on $\bp_n(I)$, so it remains to study
$f_\alpha$ for $\alpha < 1$.
Now suppose $2 \leq k \leq n$ and $\alpha < 1$. By Theorem
\ref{Tpositive}, there exists $A \in \bp_n^2(I)$ such that $f_\alpha[A]
\notin \bp_n$. Setting $B=A$, we obtain: $f_\alpha[A] + f_\alpha[B] -
f_\alpha[A+B] = (2 - 2^\alpha) f[A] \notin \bp_n$. It follows that
$f_\alpha$ is not sub-additive on $\bp_n^k(I)$ for $\alpha < 1$. This
settles the assertion for $2 \leq k \leq n$.

The last case is if $k=1$ and $\alpha < 1$. For ease of exposition, the
analysis in this case is divided into several sub-cases:\medskip

\noindent {\bf Sub-case 1: $\alpha \in (0,1)$.}
Given $0 < \epsilon < 1$ and $v \in (0,1)^n$, apply Taylor's theorem
entrywise to $g_\alpha[\epsilon v v^T]$, where $g_\alpha(x) =
(1+x)^\alpha$ as above. We obtain:
\[ f_\alpha[{\bf 1}_{n \times n} + \epsilon v v^T] - {\bf 1}_{n \times n}
- f_\alpha[\epsilon v v^T] = \epsilon \alpha v v^T  -\epsilon^\alpha
f_\alpha[v v^T] + O(\epsilon^2) C, \]

\noindent where $C = C(v)$ is a $n \times n$ matrix that is independent
of $\epsilon$. By Corollary \ref{Cindependent} with $F(x) = \epsilon
\alpha x - \epsilon^\alpha x^\alpha$ and $m = 2 \leq n$, there exist $u
\in (0,1)^n$ and $x_\alpha \in \R^n$ such that
\[ x_\alpha^T \left( f_\alpha[{\bf 1}_{n \times n} + \epsilon u u^T] -
{\bf 1}_{n \times n} - f_\alpha[\epsilon u u^T] \right) x_\alpha =
\epsilon \alpha + O(\epsilon^2) x_\alpha^T C x_\alpha, \]

\noindent which is positive for $\epsilon > 0$ small enough. Therefore
$f_\alpha$ is not sub-additive on $\bp_n^1([0,\infty))$ for $\alpha \in
(0,1)$.\medskip

\noindent {\bf Sub-case 2: $\alpha = 0$.}
To see why $f_0$ is indeed sub-additive on $\bp_n^1([0,\infty))$, given a
subset $S \subset \{ 1, \dots, n \}$ we define ${\bf 1}_S$ to be the
matrix with $(i,j)$ entry $1$ if $i,j \in S$ and $0$ otherwise. Now given
$A = (a_{ij}) \in \bp_n^1([0,\infty))$, define $S(A) := \{ i : a_{ii}
\neq 0 \}$. Then $f_0[A] = {\bf 1}_{S(A)}, f_0[B] = {\bf 1}_{S(B)}$ for
$A,B \in \bp_n^1([0,\infty))$, and hence by inclusion-exclusion, $f_0[A]
+ f_0[B] - f_0[A+B] = {\bf 1}_{S(A) \cap S(B)} \in \bp_n^1([0,\infty))$.
Thus $f_0$ is sub-additive on $\bp_n^1([0,\infty))$ as claimed.\medskip

\noindent {\bf Sub-case 3: $\alpha < 0, n \geq 3$.}
This part follows by an argument similar to Sub-case 1.
Given $0 < \epsilon < 1$ and $v \in (0,1)^n$, apply Taylor's theorem
entrywise to $g_\alpha[\epsilon v v^T]$, where $g_\alpha(x) :=
(1+x)^\alpha$. We obtain:
\[ f_\alpha[{\bf 1}_{n \times n}] + f_\alpha[\epsilon v v^T] -
f_\alpha[{\bf 1}_{n \times n} + \epsilon v v^T]
= \epsilon^\alpha f_\alpha[v v^T] - \epsilon \alpha v v^T 
- \epsilon^2 \frac{\alpha(\alpha-1)}{2} (v^{\circ 2})(v^{\circ 2})^T
+ O(\epsilon^3) C, \]

\noindent where $C = C(v)$ is a $n \times n$ matrix that is independent
of $\epsilon$. By Corollary \ref{Cindependent} with $F(x) =
\epsilon^\alpha x^\alpha - \epsilon \alpha x - \epsilon^2 x^2$ and $m = 3
\leq n$, there exist $u \in (0,1)^n$ and $x_\alpha \in \R^n$ such that
\[ x_\alpha^T \left( f_\alpha[{\bf 1}_{n \times n}] + f_\alpha[\epsilon u
u^T] -f_\alpha[{\bf 1}_{n \times n} \epsilon u u^T] \right) x_\alpha =
-\epsilon^2 \frac{\alpha(\alpha-1)}{2} + O(\epsilon^3) x_\alpha^T C(u)
x_\alpha, \]

\noindent which is negative for $\epsilon > 0$ small enough. Therefore
$f_\alpha$ is not sub-additive on $\bp_n^1([0,\infty))$ for $\alpha < 0$
and $n \geq 3$.
\medskip

\noindent {\bf Sub-case 4: $\alpha < 0, (n,k) = (2,1)$.}
The bulk of the work in classifying the sub-additive entrywise powers
$f_\alpha$ lies in the remaining case of $\bp_2^1$ with $\alpha < 0$. We
first show that $f_\alpha$ is sub-additive on $\bp_2^1([0,\infty))$ for
all $\alpha < 0$. Setting $A := (a,b)^T (a,b)$ and $B := (c,d)^T (c,d)$,
the problem translates to showing that
\[ (f_\alpha(a^2) + f_\alpha(c^2) - f_\alpha(a^2 + c^2))
\cdot (f_\alpha(b^2) + f_\alpha(d^2) - f_\alpha(b^2 + d^2)) \geq (
f_\alpha(ab) + f_\alpha(cd) - f_\alpha(ab+cd))^2. \]

\noindent Note that if any of $a,b,c,d = 0$ then the inequality is clear.
Thus we may assume $a,b,c,d > 0$. Now define 
\[ f(x,y) := x^\alpha + y^\alpha - (x+y)^\alpha, \qquad g(x,y) := \log
f(e^x,e^y). \]

\noindent Then proving the above inequality is equivalent to showing that
$(g(x_1,y_1) + g(x_2,y_2))/2 \geq g((x_1+x_2)/2, (y_1+y_2)/2)$, i.e.,
that $g$ is midpoint-convex on $\R^2$. Since $g$ is smooth, it suffices
to show that its Hessian $H_g(x,y)$ is positive semidefinite at all
points in $\R^2$. A straightforward but longwinded computation
demonstrates that $\det H_g(x,y) = 0$ for all $x,y \in \R^2$. Thus it
suffices to show that $g_{xx}$ is nonnegative on $\R$ (and the result for
$g_{yy}$ follows by symmetry). We now compute, setting $E := e^x + e^y$
for notational convenience:
\begin{align*}
&\ f(x,y)^2 g_{xx}(x,y)\\
= &\ (e^{\alpha x} + e^{\alpha y} - E^\alpha) (\alpha^2 e^{\alpha x} -
\alpha e^x E^{\alpha-1} - \alpha(\alpha- 1) e^{2x} E^{\alpha-2}) -
\alpha^2 (e^x E^{\alpha-1} - e^{\alpha x})^2\\
= &\ \alpha^2 e^{\alpha(x+y)} E^{\alpha-2}(E^{2-\alpha} - (e^{(2-\alpha)x} +
e^{(2-\alpha)y})) + \alpha e^{x+y} E^{\alpha-2}(E^\alpha - (e^{\alpha x}
+ e^{\alpha y})).
\end{align*}

\noindent Note that the difference in the first term is nonnegative
because $x^{2-\alpha}$ is super-additive, while the difference in the
second term is nonpositive because $x^\alpha$ is sub-additive. Thus both
terms are nonnegative, which concludes the proof for $f_\alpha$.\medskip

\noindent {\bf (b) Sub-additivity for $\phi_\alpha$.}
By considering the matrices $A = u u^T, B = v v^T$ with $u = (1,1)^T$
and $v = (1,-1)^T$, it immediately follows that $\phi_\alpha$
is not Loewner sub-additive on $\bp_2^1(\R)$, and hence not sub-additive
on $\bp_n^k(\R)$ for all $(n,k)$.\medskip

\noindent {\bf (c) Sub-additivity for $\psi_\alpha$.}
First note that $\psi_1(x) = x$ is sub-additive on $\bp_n^k(\R)$ for all
$(n,k)$. It is also not difficult to show that $\psi_0$ is sub-additive
on $\bp_2^1(\R)$. 
Using part (a), it remains to prove that $\psi_0$ is not
sub-additive on $\bp_n^1(\R)$ for $n>2$, and $\psi_\alpha$ is not
sub-additive on $\bp_2^1(\R)$ for $\alpha < 0$.

To see why $\psi_0$ is not sub-additive on $\bp_n^1(\R)$ for $n \geq 3$,
use the following three-parameter family of rank one counterexamples:
\[ (A(a,b,c) := (-a,c,c)^T(-a,c,c), \ B(a,b,c) := (c,-b,c)^T
(c,-b,c)), \qquad 0 < a < b < c. \]

It remains to show that $\psi_\alpha$ is not sub-additive on
$\bp_2^1(\R)$ for any $\alpha < 0$. Let $A := (1,-1)^T (1,-1)$ and $B :=
(1,1/2)^T (1,1/2)$. Then,
\[ \psi_\alpha[A] + \psi_\alpha[B] - \psi_\alpha[A+B] = \begin{pmatrix} 2
- 2^\alpha & -1 + 2^{1-\alpha}\\ -1 + 2^{1-\alpha} & 1 + (1/4)^\alpha -
(5/4)^\alpha \end{pmatrix} =: C_\alpha. \]

\noindent We claim that $\det C_\alpha < 0$ for all $\alpha < 0$, which
shows that $\psi_\alpha$ is not sub-additive on $\bp_2^1(\R)$ and
completes the proof. To see why the claim holds, compute for $\alpha < 0$:
\[ 4^\alpha (2 - 2^\alpha) \det C_\alpha = 4^\alpha + 2^\alpha - 1 -
5^\alpha. \]

\noindent Note that the function $f_\alpha(x) = x^\alpha$ is convex on
$(0,\infty)$ for $\alpha < 0$, so Jensen's inequality yields:
\[ 2^\alpha = f_\alpha \left( \frac{3}{4} \cdot 1 + \frac{1}{4} \cdot 5
\right) < \frac{3}{4} + \frac{1}{4} \cdot 5^a, \qquad
4^\alpha = f_\alpha \left( \frac{1}{4} \cdot 1 + \frac{3}{4} \cdot 5
\right) < \frac{1}{4} + \frac{3}{4} \cdot 5^a. \]

\noindent Adding the two inequalities shows that $\det C_\alpha < 0$ for
$\alpha < 0$, and the proof is complete.
\end{proof}

\section{Concluding remarks}\label{S5}

We conclude by discussing the following questions that naturally arise
from the above analysis:
\begin{enumerate}
\item Is it possible to find matrices of rank exactly $k$ (for some $1
\leq k \leq n$) for which a non-integer power less than $n-2$ is not
Loewner positive? Similar questions can also be asked for monotonicity,
convexity, and super/sub-additivity.

\item Can we compute the Hadamard critical exponents for convex
combinations of the two-sided power functions $\phi_\alpha, \psi_\alpha$?
\end{enumerate}

\noindent We answer all these questions in the remainder of this section.

\subsection{Examples of fixed rank positive semidefinite matrices.}

In the previous sections, several matrices in $\bp_n^k(I)$ were
constructed, for which specific non-integer Hadamard powers were not
Loewner positive (or monotone, convex, or super/sub-additive).  Recall
that $\bp_n^k(I)$ consists of matrices of rank $k$ {\it or less};
however, all of the examples above were in fact of rank 2 or rank 1. A
natural question is thus if there are higher rank examples for which the
Loewner properties are violated. This is important since in applications,
when entrywise powers are applied to regularize covariance/correlation
matrices, the rank corresponds to the sample size and is thus known.
Hence, failure to maintain positive definiteness only for rank $2$
matrices might not have serious consequences in practice.
The following result strengthens Theorems \ref{Tpositive},
\ref{Tmonotone}, \ref{Tconvex}, and \ref{Tsuper} by constructing higher
rank counterexamples from lower rank matrices. 

\begin{proposition}
Let $0 < R \leq \infty$, $I = [0,R)$ or $(-R,R)$, and $f : I \to \R$ be
continuous. Suppose $1 \leq l < k \leq n$ are integers, and $A, B \geq 0$
are matrices in $\bp_n(I)$ such that $\rk A = l$ and one of the following
properties is satisfied: 
\begin{enumerate}[(a)]
\item $f[A] \not\in \bp_n$;
\item $A \geq B \geq 0$ and $f[A] \not\geq f[B]$; 
\item $A \geq B \geq 0$ and $f[\lambda A + (1-\lambda)B] \not\leq \lambda
f[A] + (1-\lambda) f[B]$ for some $0 < \lambda < 1$;
\item $f[A+B] \not\geq f[A] + f[B]$.
\item $f[A+B] \not\leq f[A] + f[B]$.
\end{enumerate}
Then there exist $A', B' \geq 0$ such that $\rk A' = k$ and the same
property holds when $A,B$ are replaced by $A',B'$ respectively.
\end{proposition}

\noindent Note that the special cases of $l=1,2$ answer Question (1)
above. 

\begin{proof}
We show the result for property $(b)$ monotonicity; the analogous results
for $(a)$ positivity, $(c)$ convexity, $(d)$ super-additivity, and $(e)$
sub-additivity are shown similarly. Suppose $A \geq B \geq 0$ and $\rk A
= l$, but $f[A] \not\geq f[B]$. Then there exists a nonzero vector $v \in
\R^n$ such that $v^T f[A] v < v^T f[B] v$.
Now write $A = \sum_{i=1}^l \lambda_i u_i u_i^T$ where $\lambda_i \neq 0$
and $u_i$ are the nonzero eigenvalues and eigenvectors respectively.
Extend the $u_i$ to an orthonormal set $\{ u_1, \dots, u_k \}$, and
define $C := \sum_{i=l+1}^k u_i u_i^T$. Clearly, $A + \epsilon C \geq B +
\epsilon C \geq 0$ and $A + \epsilon C, B + \epsilon C \in \bp_n(I)$ for
small $\epsilon > 0$. Since
\[ 0 > v^T f[A] v - v^T f[B] v = \lim_{\epsilon \to 0^+} v^T (f[A +
\epsilon C] - f[B + \epsilon C]) v \]

\noindent and $f$ is continuous, there exists small $\epsilon_0 > 0$ such
that $f[A + \epsilon_0 C] \not\geq f[B + \epsilon_0 C]$. Now setting $A'
:= A + \epsilon_0 C, B' := B + \epsilon_0 C$ completes the proof, since
$A' \in \bp_n^k(I)$.
\end{proof}

\subsection{Convex combinations of power functions}

Another related question to Theorem \ref{thm:main} would be to compute
Hadamard critical exponents for the function $(\phi_\alpha +
\psi_\alpha)/2$, which equals $x^\alpha$ on $[0,\infty)$ and $0$ on
$(-\infty,0)$. More generally, consider $f^\lambda_\alpha := \lambda
\phi_\alpha + (1-\lambda) \psi_\alpha$ for $\lambda \in [0,1]$. Theorem
\ref{thm:main} already yields much information about the sets
$\calh_J^{f^\lambda}(n,k)$ for all $n \geq 2$, $1 \leq k \leq n$, and $J
\in \{$positivity, monotonicity, convexity, super/sub-additivity$\}$. In
particular, it follows from Theorem \ref{thm:main} that given $J, n, k$,
the Hadamard critical exponents for $f^\lambda$ are the same as in
Corollary \ref{Ccritexp}. Namely, the critical exponents are: 
\begin{equation}
CE_{pos}^{f^\lambda}(n,k) = n-2, \qquad
CE_{mono}^{f^\lambda}(n,k) = n-1, \qquad
CE_{conv}^{f^\lambda}(n,k) = n, \qquad \forall\ 2 \leq k \leq n,
\end{equation}

\noindent and are $0,0,1$ respectively, if $k=1 < n$.
To see why, first note that for $2 \leq k \leq n$, the sets of Loewner
positive, monotone, or convex maps are closed under taking convex
combinations. Thus the critical exponents for positivity, monotonicity,
and convexity are at most $n-2, n-1, n$ respectively. Next, the function
$f_\alpha^\lambda$ reduces to $f_\alpha$ on $[0,\infty)$ for all $\lambda
\in [0,1]$. Thus using the specific matrices described in the proofs of
Theorems \ref{thm:fitz_horn_fractional} and \ref{thm:HF_monotone} and
\cite[Lemma 5.2]{Hiai2009}, we observe that for all $\lambda \in [0,1]$
and non-integer powers $0<\alpha < n-2, n-1$, or $n$, the function
$f_\alpha^\lambda$ is not Loewner positive, monotone, or convex
respectively. The $k=1$ case is handled similarly. 

We now consider the super/sub-additivity of $f^\lambda_\alpha$. By
Theorem \ref{Tsuper} and similar arguments as above, one verifies that
$CE_{super}^{f^\lambda}(n,k) = n$ for $1 \leq k \leq n$. We next claim
that $f^\lambda_\alpha$ is never sub-additive for $0 < \lambda < 1$,
$\alpha > 0$, and $1 \leq k \leq n$. To see why, note that
$f^\lambda_\alpha$ is not sub-additive on $\bp_n^k(\R)$ for $\alpha \in
(0,\infty) \setminus \{1\}$ since $f_\alpha$ is not sub-additive on
$\bp_n([0,\infty))$. If $\alpha = 1$, using $A = {\bf 1}_{2 \times 2}
\oplus {\bf 0}_{(n-2) \times (n-2)}$ and $B = \begin{pmatrix} 1 & -1\\-1
& 1\end{pmatrix} \oplus {\bf 0}_{(n-2) \times (n-2)}$ shows that
$f^\lambda_1$ is not sub-additive for $0 < \lambda < 1$. 

\bibliographystyle{plain}
\bibliography{biblio}

\end{document}